\documentclass[10pt]{amsart}
\usepackage{a4wide}
\usepackage{float,epsfig,hyperref,graphicx,graphics,mathrsfs,dsfont}
\usepackage{color,cite,epsfig,epstopdf,amsmath,amssymb,amsfonts}
\theoremstyle{definition}
\newtheorem{dfn}{Definition}[section]
\newtheorem{exm}{Example}[section]
\newtheorem{thm}{Theorem}[section]
\newtheorem{cor}{Corollary}[section]

\newtheorem{lm}{Lemma}[section]
\newtheorem{nt}{Note}[section]

\newcommand{\myproduct}[2]{#1=\prod_{k=1}^m#2}
\newcommand{\myseries}[3]{#1\leq #2\leq #3}
\newcommand{\myvec}[3]{#1=(#2_1,#2_2,\ldots,#2_#3)}
\newcommand{\norm}[2]{\lVert #1 \rVert_#2}

\newcommand{\profunc}[1]{#1_1\times#1_2\times\cdots\times#1_m}
\newcommand{\myfunction}[3]{#1:#2\rightarrow#3}
\newcommand{\prodifs}[1]{#1_{i_1i_2\cdots i_m}}
\newcommand{\prodpoint}[2]{#1_{1,#2_1},#1_{2,#2_2},\ldots,#1_{m,#2_m}}
\title{Study of attractors and fractal functions on the product spaces and Dimensional aspects}

\author{Alamgir Hossain}
\address{Department of Mathematics, Presidency University, 86/1, College Street, Kolkata, 700 073, West Bengal, India}
\curraddr{}
\email{hossain4791@gmail.com}

\begin{document}
	\pagestyle{plain}	
	\maketitle	
	
	\begin{abstract}
		In this paper, the product of the Hausdorff metric on the product space is defined and the equivalency between the product Hausdorff metric and the Hausdorff metric on the product space is established. The finite product of the iterated function systems (IFS) on the product space is considered and the relation between the attractor of the product IFS and the attractors of the co-ordinate IFSs is studied. Dimension bounds of the  homogeneous and inhomogeneous attractors on the product space is established. Also, the product fractal interpolation function on the higher dimensional space is constructed.
	\end{abstract}
	\section{Introduction}
	A fractal can be described as an object which exhibits interesting features on a large range of scales. In pure mathematics,  the middle third Cantor set, the Sierpinski triangle, the boundary of the Mandelbrot set, and the von Koch snowflake are archetypal examples and, in `real life', examples include the surface of a lung, the horizon of a forest, and the distribution of stars in the galaxy. The term `fractal' was first used by mathematician Benoit Mandelbrot \cite{FGN} in 1975. Mandelbrot based it on the Latin fractus, meaning `broken' or `fractured', and used it to extend the concept of theoretical fractional dimensions to geometric patterns in nature. In mathematics, an iterated function system (IFS) is a method of constructing fractal. IFS fractals are more related to set theory than fractal geometry. The notion of dimension is central to fractal geometry. A dimension is a (usually non-negative real) number which gives geometric information concerning how the object in question fills up the space on small scales. There are many distinct notions of dimension and one of the joys (and central components) of the subject is in understanding how these notions relate to each other.  One way of constructing new fractal from old is by forming Cartesian products in the higher dimension spaces. Yimin xiao \cite{hpdoncp}, J. M. Marstrand \cite{tdocpset}, A. S. Besicovitch and P. A. P. Moran \cite{dotsset} and J. D. Howroyd \cite{PD} studied some certain results on various dimensions of the product set in the product spaces.   \\
	
	Formally, an IFS is a collection of finite set of contraction mappings on a complete metric space. Hutchinson \cite{AOIFS} in 1981 showed that for a complete metric space $X$, such a system of functions has a unique nonempty compact fixed set S. This set is called an attractor (or some time it is called homogeneous attractor) of an IFS. In \cite{ifsabcof}, M. F. Barnsley and S. Demko introduce a new concept of an IFS with condensation set. They have shown that this systems also have a unique fixed set. This set is called inhomogeneous attractor of an IFS with condensation. In general, the inhomogeneous attractor is not similar to homogeneous attractor. But, if the condensation set is empty, then it  coincide with homogeneous attractor of an IFS. Nina Snigireva \cite{Inssm} has given a nice relation between the inhomogeneous and homogeneous attractors of an IFS with condensation set. In the Lemma~\ref{Rbhia}, we proved this relation in a different way, which is not similar to the proof given in \cite{Inssm}. Indeed, many fractals that occur in practice, for example an attractor of an IFS, are product or, at least, are locally product-like. But in general the attractor in higher dimensional spaces are difficult to study or even to visualize. Thus study of geometrical properties of the attractor of an IFS in a higher dimensions spaces with respect to the co-ordinate IFSs is relevant. P. F. Duvall, L. S. Husch \cite{afifs} introduced the product of two hyperbolic IFSs and studied their properties. In \cite{onnmifs}, Rinju Balu and Mathew introduced (n,m)-IFSs, which generate attractors that lie in the product space. Also, in \cite{spoifs}, R. K. Aswathy and S. Mathew studied the separation properties of the products of IFSs. Interpolating a non-differential function which passes through some specific data points on a plane is very engrossing. M. F. Barnsley in 1986, first introduced the fractal interpolation function in his airtle \cite{Ftmapp}. M. A. Navascues \cite{ftaf,FSH,FFONS} studied the approximation theory of the fractal interpolation functions. P. Massopust \cite{fhsandawg} defined the fractal hypersurfaces on higher dimension space. M. N. Akhtar, et al. \cite{bdoaff,bdoaffwvsf} studied the box-dimension of an $\alpha$-fractal function.\\
	
	In this article, we discussed the finite products of hyperbolic IFSs in a different setting. Using the similar idea, we defined the product fractal function on the higher dimension space. Also, we established the approximation properties of the product fractal function.  We discussed about some preliminaries and existing results on dimensions in Section~\ref{preli} and Section~\ref{exrofdofps}. In Section~\ref{pohm}, we introduced the product Hausdorff metric and we  have shown that the equivalency between the product Hausdorff metric and the Hausdorff metric on the product space and also, we studied about the product of IFSs and homogeneous attractor of the product IFS and we provided result on the dimension of the homogeneous attractor of the product IFS. In Section~\ref{secforina}, we discussed about  product of inhomogeneous attractor  of the product IFS with condensation set and studied the dimension of the inhomogeneous attractor of the product IFS. In Section~\ref{secpfs}, we established some results of product fractal function on the higher dimension space. 
	
	\section{Some notations and preliminary results}
	\label{preli}
	In this section, we recall some basic definitions of product metric,Cartesian product of the functions, Hausdorff metric, iterated function system and Hutchinson operator that are needed in the sequel. We recommend the reader to consult the references \cite{TMS, Topo, FG, FE} for further details.
	\subsection{Product metric}
	Let $(X_k,d_k)$ for $k=1,2,\ldots,m,$  be the metric spaces and $X=\prod_{k=1}^{m}X_k$, be the Cartesian product of $X_k$. Then the product metric on $X$, is denoted by $d_0$ and is defined by
	$$d_0(\vec{x},\vec{y})=\left\{\sum_{k=1}^{m}(d_k(x_k,y_k))^2\right \}^{\frac{1}{2}} \quad  \mbox{for all} \quad \vec{x}=(x_1,x_2,\ldots,x_m),\vec{y}=(y_1,y_2,\ldots,y_m)\in X.$$
	If each metric space $(X_k,d_k)$ is complete then $(X,d_0)$ is also complete (see \cite{IFAA}). Let $\pi_k:X\rightarrow X_k$ be the projection map on $X_k$ for $k=1,2,\ldots,m$.
	\subsection{Cartesian product of the functions} Let $f_k:X_k\rightarrow Y_k$ for $\myseries{1}{k}{m}$, be the given functions, then product function from $\myproduct{X}{X_k}$ to $\myproduct{Y}{Y_k}$, is denoted by $\profunc{f}$, is defined by  $$\profunc{f}(\vec{x})=(f_1(x_1),f_2(x_2),\ldots,f_m(x_m))\quad\mbox{for}~\myvec{\vec{x}}{x}{m}\in X.$$
	If each $f_k$ for $\myseries{1}{k}{m}$, are continuous function, then product function $\profunc{f}$ is also continuous function with respect to product metric.
	\subsection{Hausdorff metric}
	Let $\delta>0$ and $(Y,d)$ be a metric space. Recall that the $\delta$-neighbourhood of a subset $A$ of $Y$ is the set of points within distance $\delta$ of $A$, that is $A_\delta=\{y\in Y: d(y,a)\leq\delta~\mbox{for some}~a \in A\}.$ Let $\mathscr{H}(Y)$ denote the collections of all nonempty compact subsets of $Y$. Then the Hausdorff distance (see\cite{FG}) between the sets $A$ and $B$ in $\mathscr{H}(Y)$, denoted by $h$, is defined by
	\begin{equation*}
		h(A,B)=\inf\left\{\delta:A\subset B_\delta~\mbox{and}~B\subset A_\delta\right\}.
	\end{equation*}
	If $(Y,d)$ is a complete metric space, then $\left ( \mathscr{H}(Y), h \right)$ is also a complete metric space (see \cite{FG,FE}).
	\begin{lm} \cite{FE}
		If $\left\{C_i\right\}_{i\in\Lambda}$, $\left\{D_i\right\}_{i\in\Lambda}$ are two finite collections of sets in $\left ( \mathscr{H}(Y), h \right)$, then $$h\left(\cup_{i\in\Lambda} C_i,\cup_{i\in\Lambda} D_i\right)\leq \max_{i\in\Lambda}h(C_i,D_i).$$
		\label{hmet}
	\end{lm}
	\subsection{Homogeneous attractor}
	It is known that a mapping $f:Y\rightarrow Y$ on a complete metric space $(Y,d)$ is called contraction map if there exist a  constant $0\leq c<1,$ such that
	$$d(f(x),f(y))\leq c~ d(x,y)\quad \mbox{ for all} \quad x,y\in Y.$$ The constant $c$ is called the contractivity factor of $f$. If the equality holds in the above inequality, then $c$ is called the contraction similarity of $f$. Banach proved the following principle in 1922:
	
	\begin{thm}\cite{FPT,TMS} Let $f:Y\rightarrow Y$ be contraction mapping on a complete metric space $(Y,d)$. Then the map $f$ has a unique fixed point $y_0\in Y$. Moreover, for any point $y\in Y$, the sequence $\left\{f^n(y)\right\}_{n=1}^\infty$ converges to $y_0$. \label{fpt}
	\end{thm}
	
	\begin{dfn}
		A (hyperbolic) iterated function system  consists of a complete metric space $(Y,d)$ together with a finite set of contraction mappings $f_i: Y\rightarrow Y$ with respect to the contractivity factors $c_i$ for $i=1,2,\ldots,N$, where $N\in\mathbb{N}$.	
	\end{dfn}
	Now from Lemma~\ref{hmet}, we get the following theorem.
	
	\begin{thm}\cite{FG}
		Let $\left\{(Y;f_i):i=1,2,\ldots,N\right\}$ be an IFS, where $f_i$'s are contraction maps with contractivity factors $c_i$ for $i=1,2,\ldots,N$. Then the Hutchinson operator $f:\mathscr{H}(Y)\rightarrow \mathscr{H}(Y)$ defined by $f(B)=\bigcup_{i=1}^Nf_i(B)$ for $B\in \mathscr{H}(Y)$, is also a contraction map with respect to the Hausdorff metric $h$ with contractivity factor $c=\max_{1\leq i \leq N} \{c_i\}$.
		\label{conth}
	\end{thm}
	If $(Y,d)$ is a complete metric space, then by the Theorem~\ref{fpt}, $f$ has a unique fixed point $B\in\mathscr{H}(Y)$. This fixed point is called an \textbf{attractor} or \textbf{homogeneous attractor} of the IFS $\left\{(Y;f_i):i=1,2,\ldots,N\right\}$. Throughout the sequel, we assume that all the metric spaces are complete.
	
	\subsection{Inhomogeneous attractor}
	\label{rbhainhatt}
	Let $\{(Y;f_i):i=1,2,\ldots,N\}$ be the IFS on $Y$, with contractivity factors $r_i, i=1,2,\ldots,N$. Let $\mathbb{N}_N=\{1,2,\ldots,N\}$ and $\mathbb{N}_N^*=\bigcup_{n=1}^{\infty}\mathbb{N}_N^n$, where $\mathbb{N}_N^n$ denotes the family of all strings $\textbf{i}=i_1i_2\ldots i_n$ of length $n$ with entries $i_j\in \mathbb{N}_N$. Finally, for $\textbf{i}=i_1i_2\ldots i_n\in \mathbb{N}_N^n$, we write
	$$f_{\textbf{i}}=f_{i_1}f_{i_2}\cdots f_{i_n}\quad\mbox{and}\quad r_{\textbf{i}}=r_{i_1}r_{i_2}\cdots r_{i_n}.$$
	
	There is an another important way of making contraction mapping on $\mathscr{H}(Y)$, where $(Y,d)$ is a complete metric space.
	\begin{dfn}(see \cite{FE,Inssm})
		Let $(Y,d)$ be a metric space and  $C\in \mathscr{H}(Y)$. Define a set valued transformation $f_0:\mathscr{H}(Y)\rightarrow \mathscr{H}(Y)$ by $f_0(B)=C$ for all $B\in \mathscr{H}(Y)$. Then $f_0$ is called a \textbf{condensation transformation} and $C$ is called the associated \textbf{condensation set}.
	\end{dfn}
	\begin{nt}
		Observe that a condensation transformation $f_0:\mathscr{H}(Y)\rightarrow \mathscr{H}(Y)$ is contraction map on the space $(\mathscr{H}(Y),h)$, with contractivity factor equal to zero, and it possesses a unique fixed point, namely the condensation set.
	\end{nt}
	\begin{dfn}
		Let $\{(Y;f_i):i=1,2,\ldots,N\}$ be the hyperbolic IFS with contractivity factors $0\leq r_i<1$ and $f_0:\mathscr{H}(Y)\rightarrow \mathscr{H}(Y)$ be a condensation transformation. Then $\{(Y;f_i):i=0,1,2,\ldots,N\}$ is called a hyperbolic IFS with condensation, with contractivity factors $0\leq r_i<1$.
	\end{dfn}
	\begin{thm}(see \cite{FE,Inssm,Insas}) Let $\{(Y;f_i):i=0,1,2,\ldots,N\}$  be a hyperbolic IFS with condensation set $C$ and contractivity factors $0\leq r_i<1$. Then the transformation $W_{\theta}:\mathscr{H}(Y)\rightarrow \mathscr{H}(Y)$ defined by
		\begin{equation*}
			W_{\theta}(B)=\bigcup_{i=0}^Nf_i(B) \quad \mbox{for all} \quad B\in \mathscr{H}(Y)
		\end{equation*}
		is a contraction mapping on the complete metric space $(\mathscr{H}(Y),h)$ with contractivity factor $r=\max\{r_i\}$. Its unique fixed point $F_C\in \mathscr{H}(Y)$, obeys
		$$F_C=W_{\theta}(F_C)=\bigcup_{i=0}^Nf_i(F_C)$$ and is given by $$F_C=\lim\limits_{n\rightarrow\infty}W_{\theta}^n(B)\quad\mbox{for any}~ B\in \mathscr{H}(Y).$$
		\label{inhomconstrec}
	\end{thm}
	This unique fixed point $F_C$ is called the \textbf{Inhomogeneous attractor} (with condensation set $C$). If we take $C=\emptyset$, then the attractor $F_{\emptyset}$  coincide with homogeneous attractor of the IFS  $\{(Y;f_i):i=1,2,\ldots,N\}$.
	\begin{exm}
		Let $K$ denote the triangle with vertices $(0,0)$, $(1,0)$ and $(\frac{1}{2},1)$ and the three maps are given by
		\begin{align*}
			f_1(x,y)=(\frac{x}{2},\frac{y}{2});\quad
			f_2(x,y)=(\frac{x}{2}+\frac{1}{2},\frac{y}{2});\quad
			f_3(x,y)=(\frac{x}{2}+\frac{1}{4},\frac{y}{2}+\frac{1}{2}).
		\end{align*}
		Let $f_0(x,y)=$ small circle located inside the triangle. Then Figure~\ref{fig3} represents the inhomogeneous attractor of this IFS.	
	\end{exm}
	\begin{figure}[!ht]
		\centering
		\includegraphics[angle=0, width=0.9\textwidth]{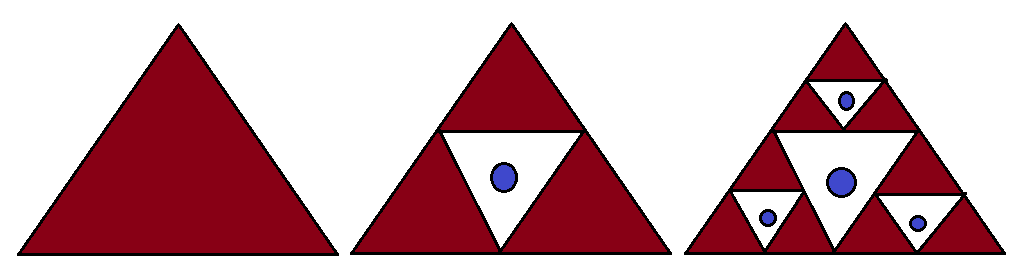}
		\caption{First three levels in the construction of Sierpinski triangle with condensation set.}
		\label{fig3}
	\end{figure}
	\subsection{Relation between homogeneous and inhomogeneous attractors}
	Define the orbital set, $O$, by
	$$O=C\cup\bigcup_{\textbf{i}\in \mathbb{N}_N^*} f_{\textbf{i}}(C),$$
	the union of the condensation set $C$, with all images of $C$ under compositions of maps in the IFS. The term orbital set was introduced by Nina Snigireva \cite{Inssm}. This set plays an important role in the structure of inhomogeneous attractors. The following lemma gives a relation in between the homogeneous and inhomogeneous attractors.
	\begin{lm}
		Let $F_C$  be the inhomogeneous attractor of an IFS $\{(Y;f_i):i=0,1,2,\ldots,N\}$ and $F_{\emptyset}$ be corresponding homogeneous attractor. Then $$F_C=O\cup F_{\emptyset}.$$
		\label{Rbhia} 	
	\end{lm}
	\begin{proof} The proof was found in \cite{Inssm} (section-3, Lemma-3.9). But here we are giving a proof which is not similar to the existing one.
		For any $B\in \mathscr{H}(Y)$, we have $W_{\theta}(B)=\bigcup_{i=0}^Nf_i(B)$. Therefore,
		\begin{align*}
			W_{\theta}(B)=&f_0(B)\cup\bigcup_{i=1}^Nf_i(B)=f_0(B)\cup f(B)\\
			W_{\theta}^2(B)=&W_{\theta}(f_0(B))\cup W_{\theta}(f(B))\\
			=&f_0(B)\cup\bigcup_{i=1}^Nf_i(f_0(B))\cup\bigcup_{i=1}^Nf_i(f(B))\\
			=&f_0(B)\cup\bigcup_{\textbf{i}\in \mathbb{N}_N}f_{\textbf{i}}(f_0(B))\cup f^2(B)\\
			W_{\theta}^3(B)=&f_0(B)\cup\bigcup_{\textbf{i}\in \mathbb{N}_N}f_{\textbf{i}}(f_0(B))\cup\bigcup_{\textbf{i}\in \mathbb{N}_N^2}f_{\textbf{i}}(f_0(B))\cup f^3(B)\\
			=&f_0(B)\cup\bigcup_{\textbf{i}\in \bigcup_{p=1}^2\mathbb{N}_N^p}f_{\textbf{i}}(f_0(B))\cup f^3(B).	
		\end{align*}
		Continuing this process, we get for any $n\in\mathbb{N}$
		\begin{equation*}
			W_{\theta}^n(B)=f_0(B)\cup\bigcup_{\textbf{i}\in \bigcup_{p=1}^{n-1}\mathbb{N}_N^p}f_{\textbf{i}}(f_0(B))\cup f^n(B).
		\end{equation*}
		In the above expression taking limit as $n\rightarrow\infty$, with respect to the Hausdorff metric, we have
		$$F_C=f_0(B)\cup\bigcup_{\textbf{i}\in \mathbb{N}_N^*}f_{\textbf{i}}(f_0(B))\cup F_{\emptyset}=C\cup\bigcup_{\textbf{i}\in \mathbb{N}_N^*}f_{\textbf{i}}(C)\cup F_{\emptyset}=O\cup F_{\emptyset}.$$
	\end{proof}
	\subsection{Fractal function} Let $I=[a,b]$ and $\Delta$ : $a=x_0<x_1<\cdots<x_N=b$ be a partition of $I$. Let $\{(x_i,y_i)\in\mathbb{R}^2:i=0,1,2,\ldots,N\}$ be the given interpolation points in $\mathbb{R}^2$. Set $I_n=[x_{n-1},x_n]$ for $n=1,2,\ldots,N$. Suppose $L_n:I\rightarrow I_n$ are contraction homeomorphisms such that 
	\begin{align}
		\label{lmaps}
		L_n(x_0)&=x_{n-1},~L_n(x_N)=x_n,\\ 
		|L_n(x)-L_n(x')|&\leq c_n|x-x'|\quad\mbox{for all}~ x,x'\in I,~\mbox{for some}~0\leq c_n<1.
	\end{align}
	Further, assume that $F_n:I\times\mathbb{R}\rightarrow\mathbb{R}$ are continuous maps satisfying
	\begin{align}
		F_n(x_0,y_0)=y_{n-1},~F_n(x_N,y_N)=y_n,
		\label{fnmaps}
	\end{align} 
	$$|F_n(x,y)-F_n(x',y)|\leq a_n|x-x'|\quad\mbox{for all}~ x,x'\in I,$$
	$$|F_n(x,y)-F_n(x,y')|\leq b_n|y-y'|\quad\mbox{for all}~ y,y'\in\mathbb{R},$$
	for some $a_n,b_n\in(-1,1)$. Define the functions $W_n:I\times\mathbb{R}\rightarrow I_n\times\mathbb{R}$ by 
	$$W_n(x,y)=(L_n(x),F_n(x,y)).$$ The maps $W_n$ satisfies the join up condition 
	$$W_n(x_0,y_0)=(x_{n-1},y_{n-1}),\quad W_n(x_N,y_N)=(x_n,y_n)\quad\mbox{for all}~\myseries{1}{n}{N}.$$
	The IFS $\{(I\times\mathbb{R};W_n):n=1,2,\ldots,N\}$ may or may not be hyperbolic with respect to Euclidean metric on $\mathbb{R}^2$. Now, we consider a metric $d_{\theta}$ on $\mathbb{R}^2$ which is equivalent to the Euclidean metric on $\mathbb{R}^2$, which is defined as follows:
	$$d_{\theta}((x,y),(x',y'))=\lvert x-x'\rvert+\theta\lvert y-y'\rvert,$$ where $\theta$ is a positive real number. Then we get the following lemma.
	\begin{lm}[see \cite{FE}]
		If $\theta=\frac{\min\{(1-|c_n|):~n=1,2,\ldots,N\}}{\max\{|a_n|:~n=1,2,\ldots,N\}}$, $a=\max\{|c_n|+\theta|a_n|:~n=1,2,\ldots,N\}$ and $b=\max\{|b_n|:~n=1,2,\ldots,N\}$, then the IFS $\{(I\times\mathbb{R};W_n):~n=1,2,\ldots,N\}$ defined as above, associated with data set $\{(x_i,y_i)\in\mathbb{R}^2:~i=0,1,2,\ldots,N\}$ is hyperbolic with respect to the metric $d_{\theta}$ with contractivity factor $\lambda=\max\{a,b\}$. \label{contlm}	
	\end{lm}
	
	The following is a fundamental result in the theory of fractal interpolation functions:
	\begin{thm}[Barnsley \cite{Ftmapp}]
		Let $C[I]$, the space of all real-valued continuous functions on $I$, be endowed with supremum norm. That is $\norm{f}{\infty}=\max\left\{|f(x)|:~x\in I\right\}$. Consider the closed subspace 
		$$C_{y_0,y_N}[I]:=\left\{f\in C[I]:f(x_0)=y_0,~f(x_N)=y_N\right\}.$$
		Then the following holds.
		\begin{enumerate}
			\item The IFS $\{(I\times\mathbb{R};W_n):n=1,2,\ldots,N\}$ has unique attractor $G(g)$ which is the graph of a continuous function $g:I\rightarrow \mathbb{R}$ satisfying $g(x_i)=y_i$ for all $\myseries{0}{i}{N}$.
			\item The function $g$ is the fixed point of the Read-Bajraktarevic (RB) operator $T:C_{y_0,y_N}[I]\rightarrow C_{y_0,y_N}[I]$ defined by 
			$$(Tf)(x)=F_n(L_n^{-1}(x), f(L_n^{-1}(x))),\quad\mbox{for}~x\in I_n;~\myseries{1}{n}{N}.$$
		\end{enumerate}
		\label{ftoff}	
	\end{thm}
	\begin{dfn}
		The function $g$ is called the fractal interpolation function (FIF) corresponding to the data set $\{(x_i,y_i)\in\mathbb{R}^2:i=0,1,2,\ldots,N\}$.
	\end{dfn}

	\subsection{Hausdorff, packing and box dimensions}
	\begin{itemize}
		\item\textbf{Hausdorff dimension:} Let $(Y,d)$ be a metric space. Then for a non-empty subset $U$ of $Y$,  define the diameter $d(U)$ of $U$ by
		$$d(U)=\sup\{d(x,y): x,y \in U \}.$$
		For $F\subseteq Y$ and $\delta>0$, a countable cover $\{U_i:i\in \mathbb{N}\}$ of $F$ is called a $\delta$-cover of $F$ with respect to the metric $d$ if $d(U_i)\leq \delta$ for all $i\in \mathbb{N}$.
		Let $s\geq 0$. For any $\delta>0$,  define
		$$H_\delta^s(F)=\inf \left \{ \sum_{i=1}^{\infty}d(U_i)^s: \{U_i\}~ \mbox{be a}    ~\delta -\mbox{cover of F} \right\},$$
		then the s-dimensional Hausdorff measure of $F$ is given by $$H^s(F)=\lim\limits_{\delta\rightarrow 0}H_\delta^s(F)$$ and the Hausdorff dimension of $F$ \cite{FG,FE}  is given by
		\begin{align*}
			\dim_HF=&\inf\left \{s\geq 0: H^s(F)=0\right\}\\
			=&\sup\left\{s : H^s(F)=\infty\right\}.
		\end{align*}
		\item\textbf{Packing dimension:} Let $$P_\delta^s(F)=\sup \left \{ \sum_{i=1}^{\infty}d(U_i)^s\right\},$$ where supremum is taking over all the collection $\{U_i\}$ of disjoint balls of radii at most $\delta$, with centres in $F$. Let $$P_0^s(F)=\lim\limits_{\delta\rightarrow 0}P_{\delta}^s(F).$$ Then s-dimensional packing measure of $F$ is given by $$P^s(F)=\inf\left\{\sum_{i}P_0^s(F_i):F\subseteq\bigcup_{i=1}^{\infty}(F_i)\right\}$$ and the packing dimension of $F$ \cite{FG}  is given by
		\begin{align*}
			\dim_PF=&\inf\left \{s\geq 0: P^s(F)=0\right\}\\
			=&\sup\left\{s : P^s(F)=\infty\right\}.
		\end{align*}
		\item\textbf{Box dimension:} Let $(Y,d)$ be a metric space and let $F$ be a compact subset of $Y$ and $N_r(F)$ be the smallest number of  balls of radius $r$ that covers $F$. Then the upper and lower box dimension of $F$ is defined by
		\begin{align*}
			\overline{\dim}_B(F)=&\limsup\limits_{r\rightarrow 0}\frac{\log N_r(F)}{-\log r}\\ \underline{\dim}_B(F)=&\liminf\limits_{r\rightarrow 0}\frac{\log N_r(F)}{-\log r}.
		\end{align*}
		If both the dimensions are equal, then the common value is called the box dimension of $F$ \cite{FG,FE} and is denoted by $\dim_BF$.
	\end{itemize}
	\subsection{Dimension of the product set}\label{exrofdofps}
	In this section  some existing results on the dimensions of the product set is given. This results are used in the sequel to  prove the main results.
	\begin{thm} (see \cite{FG})
		For any sets $E\subseteq \mathbb{R}^p$ and $F\subseteq\mathbb{R}^q$
		$$\dim_HE+\dim_HF\leq\dim_H(E\times F)\leq\overline{\dim}_B(E\times F)\leq\overline{\dim}_BE+\overline{\dim}_BF.$$
		\label{rotpset}
	\end{thm}
	\begin{thm} (see \cite{PD})
		If $X$ and $Y$ are two arbitrary metric spaces, then $$\dim_HX+\dim_HY\leq\dim_H(X\times Y)\leq\dim_HX+\dim_PY\leq\dim_P(X\times Y)\leq\dim_PX+\dim_PY$$
		\label{rbarptset}
	\end{thm}
	\begin{dfn}
		Let $S\subseteq \mathbb{R}^p$ and $r=(r_1,r_2,\ldots,r_p)\in \mathbb{R}^p$ with $r_i>0$. Define $$rS=\{(r_1x_1,r_2x_2,\ldots,r_px_p):x\in S\}.$$ The set $S$ is called
		self-affine, if $S$ is the union of N distinct subsets, each identical with $rS$ up to translation and rotation. If $r_1 = r_2 =\cdots= r_p$,
		then a self-affine set is called self-similar.
		\label{dfssset}
	\end{dfn}
	\begin{thm} (see \cite{FDIPS})
		If $E\subseteq \mathbb{R}^p$ and $F\subseteq\mathbb{R}^q$ are the self-similar sets, then $\dim_H(E\times F)=\dim_B(E\times F)=\dim_HE+\dim_HF=\dim_BE+\dim_BF.$
		\label{sslm}
	\end{thm}
	\begin{thm}(see \cite{Inssm,Ihomossbd})
		\label{doinhomoprsets}
		Let $F_C$  be the  inhomogeneous attractor of an IFS $\{(Y;f_i):i=0,1,2,\ldots,N\}$ with the condensation set $C$ and $F_{\emptyset}$ be the corresponding homogeneous attractor of it.
		\begin{enumerate}
			\item The Hausdorff and packing dimensions of the set $F_C$ is given by
			\begin{align*}
				\dim_HF_C=&\max\{\dim_HF_{\emptyset},\dim_HC\}\\
				\dim_PF_C=&\max\{\dim_PF_{\emptyset},\dim_PC\}.
			\end{align*}
			\item If the sets $f_0(F_C),f_1(F_C),\ldots,f_N(F_C)$ are all pairwise disjoint, then the upper box dimension of the set $F_C$ is  given by
			$$\overline{\dim}_B(F_C)=\max\{\overline{\dim}_B(F_{\emptyset}),\overline{\dim}_B(C)\}.$$
		\end{enumerate}
	\end{thm}
	
	\section{Homogeneous attractor on the product space and dimension bounds}
	\label{pohm}
	Let $(X_k,d_k)$,  $k=1,2,\ldots,m,$  be complete metric spaces and $(X,d_0)$ be the product metric space defined as in Section~\ref{preli}.  Let $(\mathscr{H}(X_k),H_k)$, $k=1,2,\ldots,m,$ and $(\mathscr{H}(X),H)$ be the Hausdorff metric spaces corresponding to the complete metric spaces $(X_k,d_k)$,  $k=1,2,\ldots,m,$ and $(X,d_0)$ respectively. Then $(\mathscr{H}(X_k),H_k)$, $k=1,2,\ldots,m,$ and $(\mathscr{H}(X),H)$ are also complete metric spaces (see \cite{FE}). Let us define
	\begin{equation}
		\mathscr{H}=\prod_{k=1}^m\mathscr{H}(X_k)=\left\{A_1\times A_2
		\times\cdots\times A_m:A_k\in \mathscr{H}(X_k)~\mbox{for}~k=1,2,\ldots m\right\}.
		\label{p of h}
	\end{equation}
	Since product of compact sets is compact, so all the sets in $\mathscr{H}$ are compact. Also $\mathscr{H}$ can be treat as a subset of $\mathscr{H}(X)$ via the inclusion map. Now we define a new metric $H_0$ on $\mathscr{H}$ by
	
	\begin{equation}
		H_0(A,B)=\left\{\sum_{k=1}^{m}(H_k(A_k,B_k))^2\right\}^{\frac{1}{2}}\quad \mbox{for all}~A=\prod_{k=1}^mA_k,~B=\prod_{k=1}^mB_k\in \mathscr{H}. \label{phm} 	
	\end{equation}
	The metric $H_0$, we are calling as the product Hausdorff metric on $\mathscr{H}$.  If we restrict the Hausdorff metric $H$ on $\mathscr{H}(X)$ to $\mathscr{H}$, then  $H_0$ is equivalent to $H$ (see Theorem \ref{evt}).
	Since each $(\mathscr{H}(X_k),H_k)$ is complete so, $(\mathscr{H},H_0)$ is also complete. The following result provides an equivalency between the product Hausdorff metric $H_0$ and the restriction of the Hausdorff metric $H$ on the space $\mathscr{H}$.
	\begin{thm}
		If $H'$ is the restriction of the Hausdorff metric $H$ on $\mathscr{H}$, that is $H'=H_{|\mathscr{H}\times\mathscr{H}}$, then $H_0$ is equivalent to $H'$. In particular $$\frac{1}{\sqrt{m}}H_0(A,B)\leq H'(A,B)\leq H_0(A,B)\quad \mbox{for all}~A,B\in\mathscr{H}.$$
		\label{evt}
	\end{thm}
	\begin{proof}
		Let $A=\prod_{k=1}^mA_k,~B=\prod_{k=1}^mB_k\in \mathscr{H}$. Define
		\begin{align*}
			M=&\left\{\delta : A\subseteq B_\delta~\mbox{and}~B\subseteq A_\delta\right\}\\
			M_k=&\left\{\delta : A_k\subseteq (B_k)_\delta~\mbox{and}~B_k\subseteq (A_k)_\delta\right\}\quad \mbox{for}~k=1,2,\ldots,m.
		\end{align*}
		Let $\delta_0\in M$ and $a_k\in A_k$ for some $k\in\{1,2,\ldots,m\}.$ Since each $A_k$ is non-empty,  there exists $a_p\in A_p$ for $p=1,2,\ldots,k-1,k+1,\ldots,m$, such that $\vec{a}=(a_1,a_2,\ldots,a_m)\in A\subseteq B_{\delta_0}$. Moreover, there exist $\overline{b}=(b_1,b_2,\ldots,b_m)\in B$ such that $d_0(\vec{a},\vec{b})\leq \delta_0.$ Therefore,
		$$\left\{\sum_{n=1}^{m}(d_n(a_n,b_n))^2\right \}^{\frac{1}{2}}\leq\delta_0.$$
		This implies that $d_k(a_k,b_k)\leq\delta_0.$ Therefore, $a_k\in(B_k)_{\delta_0}$. Hence it follows that $A_k\subseteq (B_k)_{\delta_0}$. Similarly, one can prove that $B_k\subseteq (A_k)_{\delta_0}$. Hence $\delta_0\in M_k$ and therefore, $$H_k(A_k,B_k)\leq \delta_0.$$ This is true for all $k=1,2,\ldots,m.$ Therefore, $$\sum_{k=1}^{m}(H_k(A_k,B_k))^2\leq m\delta_0^2.$$ Hence from (\ref{phm}), we get $$\frac{1}{\sqrt{m}}H_0(A,B)\leq\delta_0.$$ This is true for all $\delta_0\in M$. Hence
		$$\frac{1}{\sqrt{m}}H_0(A,B)\leq H'(A,B).$$
		Now, we prove the other inequality. Let $\epsilon>0.$ Then by the definition of infimum, there is a $\delta_k\in M_k,$ such that $$H_k(A_k,B_k)\leq\delta_k\leq H_k(A_k,B_k)+\epsilon\quad \mbox{for}~ k=1,2,\ldots,m.$$
		Let $\delta=\left\{\sum_{k=1}^m\delta_k^2\right\}^{\frac{1}{2}}$ and  $\vec{a}=(a_1,a_2,\ldots,a_m)\in A$. Then $a_k\in A_k$. Now $\delta_k\in M_k,$ implies $A_k\subseteq (B_k)_{\delta_k}$. Therefore, $a_k\in (B_k)_{\delta_k}$. So, there exist $b_k\in B_k$ such that $d_k(a_k,b_k)\leq \delta_k$ for $k=1,2,\ldots,m.$ Let $\vec{b}=(b_1,b_2,\ldots,b_m)$. Then $\vec{b}\in B$ and $$d_0(\vec{a},\vec{b})=\left\{\sum_{k=1}^{m}(d_k(a_k,b_k))^2\right \}^{\frac{1}{2}}\leq \left\{\sum_{k=1}^m\delta_k^2\right\}^{\frac{1}{2}}=\delta.$$ This implies that $\vec{a}\in B_\delta$. Hence $A\subseteq B_\delta$. Similarly, one can prove that $B\subseteq A_\delta.$ Therefore, $\delta\in M.$ Moreover, $$H'(A,B)\leq\delta=\left\{\sum_{k=1}^m\delta_k^2\right\}^{\frac{1}{2}}\leq \left\{\sum_{k=1}^{m}((H_k(A_k,B_k))+\epsilon)^2\right\}^{\frac{1}{2}}.$$ This is true for each $\epsilon>0$, so, by taking $\epsilon$ tends to zero, we get
		$$H'(A,B)\leq \left\{\sum_{k=1}^{m}(H_k(A_k,B_k))^2\right\}^{\frac{1}{2}}=H_0(A,B).$$
		Hence $$\frac{1}{\sqrt{m}}H_0(A,B)\leq H'(A,B)\leq H_0(A,B)\quad \mbox{for all}~A,B\in\mathscr{H}.$$
		This shows that $H_0$ and $H'$ are equivalent.
	\end{proof}
	Let $\left\{(X_k;w_{ki_k}):i_k=1,2,\ldots,N_k\right\}$ for $k=1,2,\ldots,m$, be the IFSs with contractivity factors $c_{ki_k}<1,$ of $w_{ki_k}$, and let $X=\prod_{k=1}^mX_k.$ Now define the maps $T_{i_1i_2\ldots i_m}:X\rightarrow X$ such that
	\begin{equation}
		T_{i_1i_2\ldots i_m}(\vec{x})=\left(w_{1i_1}(x_1),w_{2i_2}(x_2),\ldots,w_{mi_m}(x_m)\right) \quad \mbox{for all} \quad  \vec{x}=(x_1,x_2,\ldots,x_m)\in X,
		\label{pifs}
	\end{equation}
	where $1\leq i_k\leq N_k~\mbox{for}~k=1,2,\ldots,m.$
	Also equation (\ref{pifs}) can be written as
	\begin{equation}
		T_{i_1i_2\ldots i_m}(\vec{x})=\prod_{k=1}^mw_{ki_k}\pi_k(\vec{x})\quad \mbox{for all} \quad  \vec{x}\in X.
		\label{pho}
	\end{equation}
	Now for $\vec{x}=(x_1,x_2,\ldots,x_m), \vec{y}=(y_1,y_2,\ldots,y_m)\in X$, we have
	\begin{align*}
		d_0(T_{i_1i_2\ldots i_m}(\vec{x}),T_{i_1i_2\ldots i_m}(\vec{y}))=&\left\{\sum_{k=1}^{m}\left(d_k(w_{ki_k}(x_k),w_{ki_k}(y_k))\right)^2\right\}^\frac{1}{2}\\
		\leq&\left\{\sum_{k=1}^{m}\left(c_{ki_k}~d_k((x_k),(y_k))\right)^2\right\}^\frac{1}{2}\\
		\leq&c_{i_1i_2,\ldots i_m}~\left\{\sum_{k=1}^{m}\left(d_k((x_k),(y_k))\right)^2\right\}^\frac{1}{2},	
	\end{align*}
	where $c_{i_1i_2,\ldots i_m}=\max_{k}\left\{c_{ki_k}\right\}<1$. Therefore,
	\begin{equation}
		d_0(T_{i_1i_2\ldots i_m}(\vec{x}),T_{i_1i_2\ldots i_m}(\vec{y}))\leq c_{i_1i_2,\ldots i_m}~d_0(\vec{x},\vec{y})\quad \mbox{for all} \quad \vec{x},\vec{y}\in X.
	\end{equation}
	This shows that each $T_{i_1i_2\ldots i_m}$ are contraction map for $1\leq i_k\leq N_k~\mbox{and}~k=1,2,\ldots,m.$
	\begin{exm}
		Let, $X_1=X_2=[0,1]$ and $X=X_1\times X_2$. Also, let the maps $f_i:X_1\rightarrow X_1$ and $g_j:X_2\rightarrow X_2$ for $i,j=1,2,$ are given by
		\begin{align*}
			f_1(x)=&\frac{x}{3}, \quad\quad\quad\quad g_1(y)=\frac{y}{2}, \\
			f_2(x)=&\frac{x}{3}+\frac{2}{3},\quad\quad g_2(y)=\frac{y}{2}+\frac{1}{2}.	
		\end{align*}
		Then $T_{ij}:X\rightarrow X$ are given by
		$$T_{ij}(x,y)=(f_i(x),g_j(y))\quad \mbox{for}~i,j=1,2,$$
		and $d_0(T_{ij}(x_1,x_2),T_{ij}(y_1,y_2))\leq \frac{1}{2}d_0((x_1,x_2),(y_1,y_2))$, where $\frac{1}{2}=\max\{\frac{1}{3},\frac{1}{2}\}.$
		\label{expifs}
	\end{exm}
	Now, let $W_k$ and $T$ be the Hutchinson operators for the IFSs $\left\{(X_k;w_{ki_k}):i_k=1,2,\ldots,N_k\right\},$ for $k=1,2,\ldots,m,$ and the product IFS $\left\{(X,T_{i_1i_2\ldots i_m}):1\leq i_k\leq N_k~\mbox{for}~k=1,2,\ldots,m\right\}$ respectively. To get the relation between the Hutchinson operators $W_k$ and $T$, first we see the following Lemmas.
	\begin{lm}
		Let $A_1,B_1\subseteq X_1$ and $A_2,B_2\subseteq X_2$. Then $(A_1\cup B_1)\times(A_2\cup B_2)=(A_1\times A_2)\cup (A_1\times B_2)\cup (B_1\times A_2)\cup (B_1\times B_2).$
	\end{lm}
	\begin{proof}
		Since, $A\times (B\cup C)=(A\times B)\cup (A\times C),$  if we put $A=(A_1\cup B_1)$, $B=A_2$ and $C=B_2$, then it follows that $(A_1\cup B_1)\times(A_2\cup B_2)=(A_1\times A_2)\cup (A_1\times B_2)\cup (B_1\times A_2)\cup (B_1\times B_2).$
	\end{proof}
	Using the mathematical induction theorem in the above Lemma we get the following.
	\begin{lm}
		Let $A_{ki_k}\subseteq X_k$ for $1\leq i_k\leq N_k$, $k=1,2,\ldots,m.$ Then
		$$ \bigcup_{i_1=1}^{N_1}\bigcup_{i_2=1}^{N_2}\cdots\bigcup_{i_m=1}^{N_m}\left(A_{1i_1}\times A_{2i_2}\times\cdots\times A_{mi_m}\right)=\prod_{k=1}^{m}\left(\bigcup_{i_k=1}^{N_k}A_{ki_k}\right).$$
		\label{pl}
	\end{lm}
	\begin{thm}
		For all $B=\prod_{k=1}^mB_k\in \mathscr{H}$, $T_0(B)=\prod_{k=1}^mW_k(B_k)$, where $T_0$ is the restriction of $T$ on the space $\mathscr{H}\subseteq \mathscr{H}(X)$, that is $T_0=T_{|\mathscr{H}}$.
		\label{rbho}
	\end{thm}
	\begin{proof}
		Since $T_0(B)=\bigcup_{i_1=1}^{N_1}\bigcup_{i_2=1}^{N_2}\cdots\bigcup_{i_m=1}^{N_m}T_{i_1i_2\ldots i_m}(B)\quad \mbox{for all}~B\in\mathscr{H}.$ Then from equation (\ref{pho}), we get $$T_0(B)=\bigcup_{i_1=1}^{N_1}\bigcup_{i_2=1}^{N_2}\cdots\bigcup_{i_m=1}^{N_m}\left(\prod_{k=1}^mw_{ki_k}(B_k)\right).$$
		Therefore, using Lemma \ref{pl}, we have
		$$
		T_0(B)=\prod_{k=1}^{m}\left(\bigcup_{i_k=1}^{N_k}w_{ki_k}(B_k)\right).$$
		Hence $$T_0(B)=\prod_{k=1}^{m}W_k(B_k)\quad \mbox{for all}~B=\prod_{k=1}^mB_k\in \mathscr{H}\quad [\mbox{as}~W_k~\mbox{are Hutchinson operators.}]$$
	\end{proof}
	\begin{thm}
		The operator $T_0:\mathscr{H}\rightarrow\mathscr{H}$ is contraction mapping with respect to the metric $H_0$.
		\label{conofpho}
	\end{thm}
	\begin{proof}
		From Theorem~\ref{conth}, we get $$H_k\left(W_k(A_k),W_k(B_k)\right)\leq C_k~H_k(A_k,B_k)\quad\mbox{for all}~A_k,B_k\in\mathscr{H}(X_k),$$ where $C_k=\max\{c_{ki_k}:i_k=1,2,\ldots,N_k\}<1.$ Now for all $A=\prod_{k=1}^mA_k,~B=\prod_{k=1}^mB_k\in \mathscr{H}$, we have from above Theorem
		\begin{align*}
			H_0(T_0(A),T_0(B))=&\left\{\sum_{k=1}^{m}\left[H_k(W_k(A_k),W_k(B_k))\right]^2\right\}^\frac{1}{2}\\
			\leq&\left\{\sum_{k=1}^{m}\left[C_k~H_k(A_k,B_k)\right]^2\right\}^\frac{1}{2}\\
			\leq&C~H_0(A,B),
		\end{align*}
		where $C=\max_{k}\{C_k\}<1$.
	\end{proof}
	Since $T_0$ is a contraction map on the complete metric space $(\mathscr{H},H_0)$, by Banach fixed point theorem  $T_0$ has unique fixed point. We call it as \textbf{homogeneous product attractor} of the product IFS $\left\{(X,T_{i_1i_2\ldots i_m}):1\leq i_k\leq N_k~\mbox{for}~k=1,2,\ldots,m\right\}$. Now, we see the relation between the homogeneous product attractor and the homogeneous attractors of the co-ordinate IFSs.
	\begin{thm}
		If $F\in\mathscr{H}$ is the product attractor of $T_0$ and $F_k\in\mathscr{H}(X_k)$ are the attractors of $W_k$ for $k=1,2,\ldots,m$, then $F=\prod_{k=1}^{m}F_k.$
		\label{rbtaoifs}
	\end{thm}
	\begin{proof}
		Let $F=\prod_{k=1}^mE_k$.
		Since $F$ is the attractor of $T_0$, we have  $T_0(F)=F$. Then by Theorem \ref{rbho} we get
		\begin{equation}
			\prod_{k=1}^{m}W_k(E_k)=F.
			\label{atrre}
		\end{equation}
		Applying the projection map $\pi_l$ ($1\leq l\leq m$) on both side of the above equation (\ref{atrre}), we have $$W_l(E_l)=\pi_l(F)=E_l.$$
		But the fixed points are unique, so, $E_l=F_l.$ Therefore,
		$$F=\prod_{k=1}^{m}W_k(F_k)\quad [\mbox{using (\ref{atrre})}.]$$
		Since $F_k$ are the fixed points of $W_k$ for $k=1,2,\ldots,m.$ Hence $F=\prod_{k=1}^mF_k.$
	\end{proof}
	\begin{cor}
		If  $F=\prod_{k=1}^mF_k\subseteq X$ is the attractor of $T_0$, then $F_k\subseteq X_k$ are the attractors of $W_k$ for $k=1,2,\ldots\,m.$ Conversely, if $F_k$ are the attractors of $W_k$ for $k=1,2,\ldots\,m$, then $\prod_{k=1}^mF_k$ is the attractor of $T_0$.
		\label{corat}
	\end{cor}
	\begin{exm}
		Let $X_1=X_2=[0,1]$ and $X=X_1\times X_2$. Also, let the maps $f_i:X_1\rightarrow X_1$ for $ i=1,2$, and $g_j:X_2\rightarrow X_2$ for $j=1,2,3$, be given by
		\begin{align*}
			f_1(x)=&\frac{x}{2};\quad f_2(x)=\frac{x}{2}+\frac{1}{2};\\
			g_1(x)=&\frac{x}{6};\quad g_2(x)=\frac{x}{6}+\frac{3}{6};\quad g_3(x)=\frac{x}{6}+\frac{5}{6}.
		\end{align*}
		Then
		\begin{align*}
			T_{11}(x,y)=&\left(\frac{x}{2},\frac{y}{6}\right),\quad \quad T_{12}(x,y)=\left(\frac{x}{2},\frac{y}{6}+\frac{3}{6}\right),\quad\quad T_{13}(x,y)=\left(\frac{x}{2},\frac{y}{6}+\frac{5}{6}\right),\\
			T_{21}(x,y)=&\left(\frac{x}{2}+\frac{1}{2},\frac{y}{6}\right),\quad T_{22}(x,y)=\left(\frac{x}{2}+\frac{1}{2},\frac{y}{6}+\frac{3}{6}\right),\quad T_{23}(x,y)=\left(\frac{x}{2}+\frac{1}{2},\frac{y}{6}+\frac{5}{6}\right).
		\end{align*}
		Since  $F_1=[0,1]$ is the attractor of the IFS $\{(X_1;f_i):i=1,2\}$ and if $F_2$ is the attractor of	the IFS $\left\{(X_2;g_j):j=1,2,3\right\},$
		then $F=F_1\times F_2$ is the attractor of the product IFS $\{(X;T_{ij}):1\leq i\leq 2,~1\leq j\leq 3\}$ with $\pi_1(F)=F_1$ and $\pi_2(F)=F_2,$ see the Figure~\ref{fig1}.
		\label{affine exam}
	\end{exm}
	\begin{figure}[!ht]
		\centering
		\includegraphics[angle=0,height=0.4\textheight, width=0.4\textwidth]{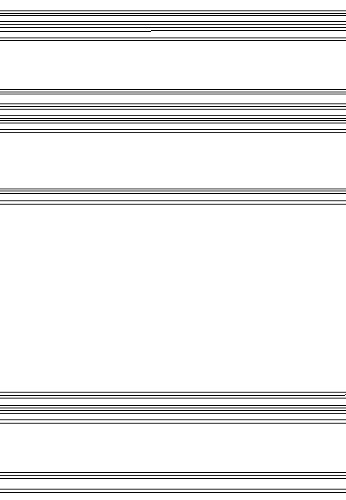}
		\caption{The product $F_1\times F_2$.} \label{fig1}
	\end{figure}
	\begin{nt}
		If all $w_{ki_k}$ are similar transformation, then $T_{i_1i_2\ldots i_m}$ are not necessary similar transformation (see the Example~\ref{affine exam}).
	\end{nt}
	\begin{dfn}
		Let $(Y,d_Y)$ is a metric space. We say that an IFS $\{(Y;f_i):i=1,2,\ldots,N\}$ satisfy the open set condition (OSC) if there exist a non-empty bounded open set $V\subseteq Y,$ such that $$\bigcup_{i=1}^Nf_i(V)\subseteq V\quad\mbox{and}\quad f_i(V)\cap f_j(V)=\emptyset\quad\mbox{for}~i\neq j.$$ 	
	\end{dfn}
	\begin{thm}
		For all $k=1,2,\ldots,m,$ the IFSs $\left\{(X_k;w_{ki_k}):i_k=1,2,\ldots,N_k\right\}$ satisfies OSC if and only if the product IFS $\left\{(X,T_{i_1i_2\ldots i_m}):1\leq i_k\leq N_k~\mbox{for}~k=1,2,\ldots,m\right\}$ satisfy the OSC.
		\label{osc}
	\end{thm}
	\begin{proof}
		Let $W_k$ and $T_0$  be defined as above. For all $k=1,2,\ldots,m,$ consider the IFSs $\left\{(X_k;w_{ki_k}):i_k=1,2,\ldots,N_k\right\}$ that satisfies OSC. Then there exist non-empty bounded open sets $V_k\subseteq X_k$ for $k=1,2,\ldots,m,$ such that $$W_k(V_k)\subseteq V_k\quad\mbox{and}\quad w_{ki_k}(V_k)\cap w_{kj_k}(V_k)=\emptyset\quad\mbox{for}~i_k\neq j_k,$$
		where $i_k,j_k\in\{1,2,\ldots,N_k\}.$ Now, let $V=\prod_{k=1}^{m}V_k.$ Then $V$ is non-empty bounded open subset of $X,$ and $$
		T_0(V)=\prod_{k=1}^mW_k(V_k)\subseteq\prod_{k=1}^mV_k=V.$$
		Also,
		\begin{align*}
			T_{i_1i_2\ldots i_m}(V)\cap T_{j_1j_2\ldots j_m}(V)=&\prod_{k=1}^mw_{ki_k}(V_k)\cap \prod_{k=1}^mw_{kj_k}(V_k)\\
			=&\prod_{k=1}^m\left(w_{ki_k}(V_k)\cap w_{kj_k}(V_k)\right)\\
			=&\emptyset \qquad \mbox{if}~i_k\neq j_k\quad\mbox{for}~i_k,j_k\in\{1,2,\ldots,N_k\}.
		\end{align*}
		Conversely, consider the product IFS $\left\{(X,T_{i_1i_2\ldots i_m}):1\leq i_k\leq N_k~\mbox{for}~k=1,2,\ldots,m\right\}$ satisfy OSC. Then there exist a non-empty bounded open set $U\subseteq X$, such that $T_0(U)\subseteq U$, and $$T_{i_1i_2\ldots i_m}(U)\cap T_{j_1j_2\ldots j_m}(U)=\emptyset\quad \mbox{if}~i_k\neq j_k\in\{1,2,\ldots,N_k\}. $$ Since projection map is a open map so, each set $\pi_k(U)$ is open in $X_k$, and also non-empty bounded. Now $T_0(U)\subseteq U$ implies
		$$\prod_{k=1}^mW_k(\pi_k(U))\subseteq U.$$
		Therefore, applying the projection map $\pi_l$ in both side of the above inequality we get
		$$W_l(\pi_l(U))\subseteq \pi_l(U).$$
		Let $i_k\neq j_k\in\{1,2,\ldots,N_k\}$. Then $$T_{l_1l_2\cdots i_k\cdots l_m}(U)\cap T_{l_1l_2\cdots j_k\cdots l_m}(U)=\emptyset\quad\mbox{for some}\quad l_p\in\{1,2,\ldots,N_p\},$$
		where, $p=1,2,\ldots,k-1,k+1,\ldots,m.$ This shows that $$w_{1l_1}(\pi_1(U))\times w_{2l_2}(\pi_2(U))\times\cdots\times (w_{ki_k}(\pi_k(U))\cap w_{kj_k}(\pi_k(U)))\times\cdots\times w_{ml_m}(\pi_m(U))=\emptyset.$$
		Since each $w_{pl_p}(\pi_p(U))\neq\emptyset$ for all $p=1,2,\ldots,m$. Hence $$w_{ki_k}(\pi_k(U))\cap w_{kj_k}(\pi_k(U))=\emptyset.$$ Therefore, for all $k=1,2,\ldots,m,$ the IFSs $\left\{(X_k;w_{ki_k}):i_k=1,2,\ldots,N_k\right\}$ satisfies OSC.
	\end{proof}
	The following theorem provides a result on the dimension of the product homogeneous attractor for self-similar cases.
	
	\begin{thm}
		Let $X_k=\mathbb{R}^{m_k}$ for some $m_k\in\mathbb{N}$, and let $\left\{(X_k;w_{ki_k}):i_k=1,2,\ldots,N_k\right\}$ for $k=1,2,\ldots,m$, be the IFSs with contraction similarities $c_{ki_k}<1$, and satisfies the open set condition (OSC) for all $k=1,2,\ldots,m$. If $F=\prod_{k=1}^mF_k$ be the attractor of the product IFS $\left\{(X,T_{i_1i_2\ldots i_m}):1\leq i_k\leq N_k~\mbox{for}~k=1,2,\ldots,m\right\}$, then $\dim_HF=\dim_BF=\sum_{k=1}^{m}s_k$, where $s_k$ is the solution of the equation $\sum_{i_k=1}^{N_k}c_{ki_k}^{s_k}=1,$ for $k=1,2,\ldots,m.$
	\end{thm}
	\begin{proof}
		By the Corollary~\ref{corat}, $F_k$ are the attractors of the IFSs $\left\{(X_k;w_{ki_k}):i_k=1,2,\ldots,N_k\right\}$ for $k=1,2,\ldots,m.$ Also, it is given that the IFSs satisfies the OSC. Hence, together it follows (see \cite{FG,FE})$$\dim_HF_K=\dim_BF_k=s_k.$$
		Since $F_k$ are self-similar sets, so, by Theorem \ref{sslm}, we have $$\dim_H\prod_{k=1}^mF_k=\dim_B\prod_{k=1}^mF_k=\sum_{k=1}^m\dim_HF_k=\sum_{k=1}^m\dim_BF_k=\sum_{k=1}^ms_k.$$
		Now by Theorem~\ref{osc}, the product IFS satisfy the OSC.
		Hence $$\dim_HF=\dim_BF=\dim_B\prod_{k=1}^mF_k=\sum_{k=1}^{m}s_k.$$  	
	\end{proof}
	\begin{exm}
		Let, $X_1=[0,1]$ and $X=X_1\times X_1$. And let, the maps $f_i:X_1\rightarrow X_1$ for $i=1,2,$ are given by
		\begin{align*}
			f_1(x)=\frac{x}{3},\quad
			f_2(x)=\frac{x}{3}+\frac{2}{3}.	
		\end{align*}
		Then middle third Cantor set $F$ (say) is the attractor of the IFS $\{(X_1;f_i):i=1,2\}$ and $F\times F$ is the attractor of the product IFS $\{(X;(f_i\times f_j)):i,j=1,2\},$ see  Figure~\ref{fig2}. In this case, $\dim_H(F\times F)=\dim_B(F\times F)=2\dim_BF=2\log 2/\log 3.$
		\begin{figure}[!ht]
			\centering
			\includegraphics[angle=0,height=0.4\textheight, width=0.5\textwidth]{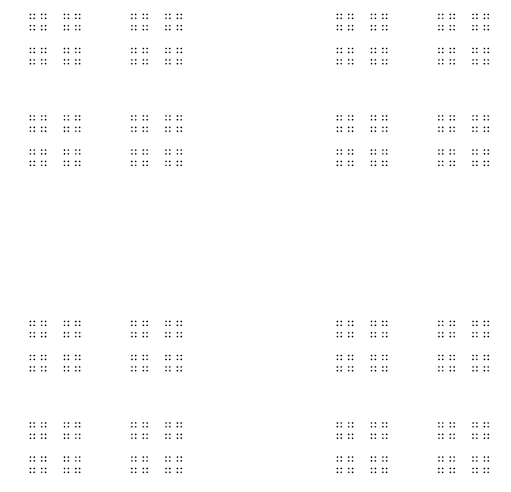}
			\caption{The product $F\times F$, where $F$ is the middle third Cantor set.} \label{fig2}
		\end{figure}
	\end{exm}
	
	\section{Inhomogeneous attractor on the product spaces and dimension bounds}\label{secforina}
	
	Let $\left\{(X_k;w_{ki_k}):i_k=1,2,\ldots,N_k\right\}$ for $k=1,2,\ldots,m$, be the IFSs with contractivity factors $c_{ki_k}<1,$ of $w_{ki_k}$, and let $C_k\in \mathscr{H}(X_k)$. Define $w_{k0}:\mathscr{H}(X_k)\rightarrow\mathscr{H}(X_k)$ be such that $w_{k0}(B)=C_k$ for all $B\in\mathscr{H}(X_k)$. Then $\left\{(X_k;w_{ki_k}):i_k=0,1,2,\ldots,N_k\right\}$ are the hyperbolic IFSs with contractivity factors $c_{ki_k}<1,$ with condensation set $C_k$.\\
	By Theorem~\ref{inhomconstrec} the transformation $W_{\theta k}:\mathscr{H}(X_k)\rightarrow\mathscr{H}(X_k)$,  given by $W_{\theta k}(B)=\bigcup_{i_k=0}^{N_k}w_{ki_k}(B)$ has a unique inhomogeneous attractor $F_{C_k}$. That is  $W_{\theta k}(F_{C_k})=F_{C_k}$. Also, for $k=1,2,\ldots,m$, each $F_{C_k}$ can be written as
	$$F_{C_k}=O_k\cup F_k,$$
	where $F_k$ is the corresponding homogeneous attractor and the orbital set $O_k$ is given by $O_k=C_k\cup\bigcup_{\textbf{j}_k\in J_k^*} w_{k\textbf{j}_k}(C_k)$, where $J_k^*=\bigcup_{n=1}^{\infty}J_k^n$, $J_k=\{1,2,\ldots,N_k\}$ and $w_{k\textbf{j}_k}=w_{kj_1}w_{kj_2}\cdots w_{kj_n}$ if $\textbf{j}_k=j_1j_2\ldots j_n\in J_k^n$.\\
	Let $X=\prod_{k=1}^mX_k$, and $\left\{(X,T_{i_1i_2\ldots i_m}):1\leq i_k\leq N_k~\mbox{for}~k=1,2,\ldots,m\right\}$ be the corresponding product IFS. Let $C=\prod_{k=1}^mC_k$. Now, we define the set valued transformations $T_{i_1i_2\ldots i_m}:\mathscr{H}\rightarrow\mathscr{H}$ such that for all $B=\prod_{k=1}^mB_k\in\mathscr{H}$
	$$T_{i_1i_2\ldots i_m}(B)=\prod_{k=1}^mw_{ki_k}(B_k),$$
	where at least one $i_p=0$, for some $1\leq p\leq m$. If $i_p=0$, for all $p=1,2,\ldots,m$, then $T_{00\ldots 0}(B)=\prod_{k=1}^mC_K=C$. Otherwise, $T_{i_1i_2\ldots i_m}(B)$ is combination of the product of some $C_k$ and $W_{ki_k}(B_k)$. Let $T_{\theta}:\mathscr{H}\rightarrow\mathscr{H}$ be defined by
	$$T_{\theta}(B)=\bigcup_{i_1=0}^{N_1}\bigcup_{i_2=0}^{N_2}\cdots\bigcup_{i_m=0}^{N_m}T_{i_1i_2\ldots i_m}(B)\quad \mbox{for all}~B\in\mathscr{H}.$$
	Then by Theorem~\ref{rbho}, we have $T_{\theta}(B)=\prod_{k=1}^mW_{\theta k}(B_k)$ for all $B=\prod_{k=1}^mB_K\in\mathscr{H}.$ Now, we get the following result for inhomogeneous cases.
	\begin{lm}
		If $\left\{(X_k;w_{ki_k}):i_k=0,1,2,\ldots,N_k\right\}$ are the hyperbolic IFSs with condensation set $C_k$, then $\left\{(X,T_{i_1i_2\ldots i_m}):0\leq i_k\leq N_k~\mbox{for}~k=1,2,\ldots,m\right\}$ is also the hyperbolic IFS with condensation set $C=\prod_{k=1}^mC_K$. Also, $T_{\theta}$ is contractive with respect to the product metric $H_0$.
	\end{lm}
	\begin{proof}
		The proof is similar to the proof of  Theorem~\ref{conofpho}.
	\end{proof}
	Therefore, the product IFS $\left\{(X,T_{i_1i_2\ldots i_m}):0\leq i_k\leq N_k~\mbox{for}~k=1,2,\ldots,m\right\}$ has an attractor $F_C$ and we call it as  \textbf{inhomogeneous product attractor}.
	\begin{thm}
		If $F_C$ is the inhomogeneous product attractor of $T_{\theta}$, then $F_C=\prod_{k=1}^mF_{C_k}$. In particular, $F_C=\prod_{k=1}^m(O_k\cup F_{k})$.
		\label{rbiatt}
	\end{thm}
	\begin{proof}
		The proof is similar to the proof of  Theorem~\ref{rbtaoifs}.
	\end{proof}
	\begin{exm}
		Let $X_1=X_2=[0,1]$ and $X=X_1\times X_2$. Consider the maps $w_{ij}:X_i\rightarrow X_i$ for $j=1,2,$ are given by
		\begin{align*}
			w_{11}(x)=&w_{21}(x)=\frac{x}{3}, \\
			w_{12}(x)=&w_{22}(x)=\frac{x}{3}+\frac{2}{3}.	
		\end{align*}
		Then middle third Cantor set $F_i$ are the attractor of the IFS $\{(X_i;w_{ij}):j=1,2\}$ for $i=1,2$, and $F_1\times F_2$ is the attractor of the product IFS $\{(X;T_{j_1j_2}=w_{1j_1}\times w_{2j_2}):j_1,j_2\in\{1,2\}\}$, which is a product of two Cantor sets (see  Figure~\ref{fig2}). Let $C_1=\{0\}\times[\frac{4}{9},\frac{5}{9}]$ and $C_2=[\frac{4}{9},\frac{5}{9}]\times \{0\}$. Now, we consider a set valued function $w_{i0}:\mathscr{H}(X_i)\rightarrow\mathscr{H}(X_i)$ such that $w_{i0}(B)=C_i$ for all $B\in\mathscr{H}(X_i)$. Let $F_{C_i}$ be the inhomogeneous attractor of the IFS $\{(X_i;w_{ij}):j=0,1,2\}$ for $i=1,2$. Then the inhomogeneous product attractor $F_C$ of the IFS $\{(X;T_{j_1j_2}=w_{1j_1}\times w_{2j_2}):j_1,j_2\in\{0,1,2\}\}$ is given by $F_C=F_{C_1}\times F_{C_2}$ (see Figure~\ref{fig4}).
	\end{exm}
	\begin{figure}[!ht]
		\centering
		\includegraphics[angle=0, width=0.9\textwidth]{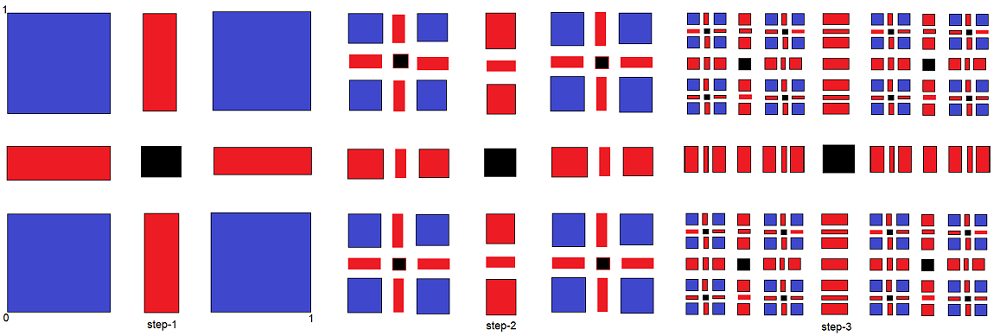}
		\caption{First three levels in the construction of inhomogeneous product attractor with condensation set.} \label{fig4}
	\end{figure}
	
	Here, $T_{00}(X)=[\frac{4}{9},\frac{5}{9}]\times[\frac{4}{9},\frac{5}{9}]$, $T_{01}(X)=[\frac{4}{9},\frac{5}{9}]\times[0,\frac{1}{3}]$, $T_{02}(X)=[\frac{4}{9},\frac{5}{9}]\times[\frac{2}{3},1]$, $T_{10}(X)=[0,\frac{1}{3}]\times[\frac{4}{9},\frac{5}{9}]$, $T_{20}(X)=[\frac{2}{3},1]\times[\frac{4}{9},\frac{5}{9}]$.
	
	\begin{lm}
		If the sets in the collection $\{w_{ki_k}(F_{C_k}):0\leq i_k\leq N_k\}$, are pairwise disjoint for $k=1,2,\ldots,m$, then the sets in the collection  $\{T_{i_1i_2\cdots i_m}(F_C):0\leq i_k\leq N_k~\mbox{for}~k=1,2,\ldots,m\}$, are pairwise disjoint. Converse is also true.	
		\label{tdict}
	\end{lm}
	\begin{proof}
		Proof  follows from the  observation of the proof of Theorem~\ref{osc}.
	\end{proof}
	The following  theorem provides a result on the dimension of the inhomogeneous product attractor.
	\begin{thm}
		For $k=1,2,\ldots,m$, let $F_{C_k}$ be the inhomogeneous attractors corresponding to the IFSs $\left\{(X_k;w_{ki_k}):i_k=0,1,2,\ldots,N_k\right\}$ with condensation set $C_k$ respectively, and $F_k$ be the corresponding homogeneous attractors. Let $F_C$ be the inhomogeneous attractor of the product IFS\\ $\left\{(X,T_{i_1i_2\ldots i_m}):0\leq i_k\leq N_k~\mbox{for}~k=1,2,\ldots,m\right\}$ with condensation set $C=\prod_{k=1}^mC_K$. Then we get the followings.
		\begin{enumerate}
			\item $\sum_{k=1}^{m}\max\{\dim_HF_k,\dim_HC_k\}\leq\dim_HF_C\leq\dim_PF_C\leq \sum_{k=1}^{m}\max\{\dim_PF_k,\dim_PC_k\}$.
			\item If the sets in the collection  $\{T_{i_1i_2\cdots i_m}(F_C):0\leq i_k\leq N_k~\mbox{for}~k=1,2,\ldots,m\}$, are pairwise disjoint, then we get  $$\overline{\dim}_BF_C\leq\sum_{k=1}^{m}\max\{\overline{\dim}_BF_k,\overline{\dim}_BC_k\}.$$
		\end{enumerate}
		\label{dotinha}
	\end{thm}
	
	\begin{proof}
		\begin{enumerate}
			\item From  Theorem~\ref{rbiatt}, we get $F_C=\prod_{k=1}^mF_{C_k}$. Then using Theorem~\ref{rotpset} and \ref{rbarptset} together it follows that $\sum_{k=1}^{m}\dim_HF_{C_k}\leq\dim_H\prod_{k=1}^mF_{C_k}=\dim_HF_C\leq\dim_PF_C\leq\sum_{k=1}^{m}\dim_PF_{C_k}$. Also, from Theorem~\ref{doinhomoprsets}, we see that  $\dim_HF_{C_k}=\max\{\dim_HF_{k},\dim_HC_k\}$ and $\dim_PF_{C_k}=\max\{\dim_PF_{k},\dim_PC_k\}$ . Therefore, $$\sum_{k=1}^{m}\max\{\dim_HF_k,\dim_HC_k\}\leq\dim_HF_C\leq\dim_PF_C\leq \sum_{k=1}^{m}\max\{\dim_PF_k,\dim_PC_k\}$$
			\item From Theorem~\ref{rotpset}, we get $\overline{\dim}_BF_C= \overline{\dim}_B\prod_{k=1}^mF_{C_k}\leq\sum_{k=1}^m\overline{\dim}_BF_{C_k}$. Since the sets in the collection $\{T_{i_1i_2\cdots i_m}(F_C):0\leq i_k\leq N_k~\mbox{for}~k=1,2,\ldots,m\}$, are pairwise disjoint, so, by Lemma~\ref{tdict}, the sets in the collection $\{w_{ki_k}(F_{C_k}):0\leq i_k\leq N_k\}$, are pairwise disjoint for $k=1,2,\ldots,m$. Hence from  Theorem~\ref{doinhomoprsets}, we get $\overline{\dim}_BF_{C_k}=\max\{\overline{\dim}_BF_k,\overline{\dim}_BC_k\}$. Therefore, $$\overline{\dim}_BF_C\leq\sum_{k=1}^{m}\max\{\overline{\dim}_BF_k,\overline{\dim}_BC_k\}$$
		\end{enumerate}
	\end{proof}
	
	\section{Construction of fractal function on the product space}\label{secpfs}
	
	Let $\{(x_{k,i_k},y_{k,i_k})\in\mathbb{R}^2:\myseries{0}{i_k}{N_k}\}$ be the prescribed data set in $\mathbb{R}^2$ such that $\Delta_k:$ $x_{k,0}<x_{k,1}<\cdots<x_{k,N_k}$ be a partition of $I_k=[x_{k,0},x_{k,N_k}]$ for $\myseries{1}{k}{m}$. Let $I_{k,i_k}=[x_{k,i_k-1},x_{k,i_k}]$ for $\myseries{1}{i_k}{N_k}$; $\myseries{1}{k}{m}$. Suppose $L_{k,i_k}:I_k\rightarrow I_{k,i_k}$ are contraction homeomorphisms such that
	\begin{equation}
		L_{k,i_k}(x_{k,0})=x_{k,i_k-1},~L_{k,i_k}(x_{k,N_k})=x_{k,i_k}, \label{equ8}
	\end{equation}
	\begin{equation}
		|L_{k,i_k}(x)-L_{k,i_k}(x')|\leq c_{k,i_k}|x-x'|\quad\mbox{for all}~ x,x'\in I_k, \label{equ9}
	\end{equation}
	where $0\leq c_{k,i_k}<1,$ for $\myseries{1}{i_k}{N_k};~\myseries{1}{k}{m}.$
	Further, assume that $F_{k,i_k}:I_k\times\mathbb{R}\rightarrow\mathbb{R}$ are continuous maps satisfying
	$$F_{k,i_k}(x_{k,0},y_{k,0})=y_{k,i_k-1},~F_{k,i_k}(x_{k,N_k},y_{k,N_k})=y_{k,i_k},$$
	$$|F_{k,i_k}(x,y)-F_{k,i_k}(x',y)|\leq |a_{k,i_k}||x-x'|\quad\mbox{for all}~ x,x'\in I_k,$$
	$$|F_{k,i_k}(x,y)-F_{k,i_k}(x,y')|\leq |b_{k,i_k}||y-y'|\quad\mbox{for all}~ y,y'\in\mathbb{R},$$
	for some $a_{k,i_k},b_{k,i_k}\in (-1,1)$, for $\myseries{1}{i_k}{N_k}$; $\myseries{1}{k}{m}$.
	Now, from Theorem~\ref{ftoff}, the IFSs $\{(I_k\times\mathbb{R};W_{k,i_k}):\myseries{1}{i_k}{N_k}\}$, where $W_{k,i_k}:I_k\times\mathbb{R}\rightarrow I_{k,i_k}\times\mathbb{R}$ are defined by $$W_{k,i_k}(x,y)=(L_{k,i_k}(x),F_{k,i_k}(x,y)),$$
	has an attractor $G(g_k)$ for $\myseries{1}{k}{m}$, which are the graph of the continuous functions $\myfunction{g_k}{I_k}{\mathbb{R}}$ satisfying $g_k(x_{k,i_k})=y_{k,i_k}$ for all $\myseries{0}{i_k}{N_k}$; $\myseries{1}{k}{m}$, where $g_k$ are the fixed point of the corresponding RB-operators $T_k:C_{y_{k,0},y_{k,N_k}}[I_k]\rightarrow C_{y_{k,0},y_{k,N_k}}[I_k]$ such that
	\begin{equation}
		(T_kf)(x)=F_{k,i_k}(L_{k,i_k}^{-1}(x),f(L_{k,i_k}^{-1}(x)))\quad\mbox{for}~x\in I_{k,i_k};~\myseries{1}{i_k}{N_k}. \label{prodrbop}
	\end{equation}  
	To simplify the notation, for $n\in\mathbb{N}$, we write 
	$$\mathbb{N}_n=\{1,2,\ldots,n\},\quad\mathbb{N}_{n,0}=\{0,1,2,\ldots,n\},\quad\partial\mathbb{N}_{n,0}=\{0,n\}.$$ 
	Let $\sigma:\mathbb{Z}\times\{0,N_1,N_2,\ldots,N_m\}\rightarrow\mathbb{Z}$ be defined by 
	\begin{equation}
		\left\{ \begin{array}{ll}
			\sigma(i,j)=i-1 & \textrm{if $j=0$};\\
			\sigma(i,j)=i & \textrm{if $j\neq 0$}.
		\end{array} \right.
	\end{equation}
	Using the above notation, we see that 
	\begin{equation}
		L_{k,i_k}(x_{k,j_k})=x_{k,\sigma(i_k,j_k)},\quad F_{k,i_k}(x_{k,j_k},y_{k,j_k})=y_{k,\sigma(i_k,j_k)}
		\label{contpoint}
	\end{equation}
	for all $i_k\in\mathbb{N}_{N_k},~j_k\in \partial\mathbb{N}_{N_k,0},~ k\in\mathbb{N}_m.$\\
	
	Now, define  the product contraction homeomorphisms $\prodifs{L}:I\rightarrow\prodifs{I}$ as
	\begin{equation}
		\prodifs{L}(\vec{x})=(L_{1,i_1}(x_1),L_{2,i_2}(x_2),\ldots,L_{m,i_m}(x_m))\quad\mbox{for all}~ \myvec{\vec{x}}{x}{m}\in I, \label{procontrac}
	\end{equation}
	where $\myproduct{I}{I_k}$ and $\prodifs{I}=\prod_{k=1}^{m}I_{k,i_k}$ for $\myseries{1}{i_k}{N_k}$; $\myseries{1}{k}{m}$.
	Then from (\ref{contpoint}), we see that 
	\begin{equation}
		\prodifs{L}(\prodpoint{x}{j})=(x_{1,\sigma(i_1,j_1)},x_{2,\sigma(i_2,j_2)},\ldots,x_{m,\sigma(i_m,j_m)})
		\label{prodofL}
	\end{equation}
	for all $i_k\in \mathbb{N}_{N_k},~j_k\in \partial\mathbb{N}_{N_k,0},~ k\in\mathbb{N}_m$ and 
	\begin{equation}
		\lVert \prodifs{L}(\vec{x})-\prodifs{L}(\vec{x'})\rVert_2\leq |\prodifs{c}|_{\infty}\lVert \vec{x}-\vec{x'}\rVert_2,
		\label{ccoflt}
	\end{equation}
	where $|\prodifs{c}|_{\infty}=\max\{c_{k,i_k}:\myseries{1}{i_k}{N_k};~\myseries{1}{k}{m}\}<1$ and $\lVert\cdot\rVert_2$ denote the Euclid norm on $\mathbb{R}^m$.\\
	
	Further, the maps $\prodifs{F}:I\times\mathbb{R}^m\rightarrow\mathbb{R}^m$ are defined by 
	\begin{equation}
		\prodifs{F}(\vec{x},\vec{y})=(F_{1,i_1}(x_1,y_1),F_{2,i_2}(x_2,y_2),\ldots,F_{m,i_m}(x_m,y_m)) \label{prod of F}
	\end{equation}
	for $\myvec{\vec{x}}{x}{m}\in I,~\myvec{\vec{y}}{y}{m}\in\mathbb{R}^m.$ Hence from (\ref{contpoint}), we get 
	\begin{equation}
		\prodifs{F}((\prodpoint{x}{j}),(\prodpoint{y}{j}))=(y_{1,\sigma(i_1,j_1)},y_{2,\sigma(i_2,j_2)},\ldots,y_{m,\sigma(i_m,j_m)})
		\label{prodofF}
	\end{equation}
	for all $i_k\in \mathbb{N}_{N_k},~j_k\in \partial\mathbb{N}_{N_k,0},~ k\in\mathbb{N}_m$ and 
	\begin{equation}
		\left. \begin{array}{ll}
			\lVert \prodifs{F}(\vec{x},\vec{y})-\prodifs{F}(\vec{x'},\vec{y})\rVert_2\leq |\prodifs{a}|_{\infty}\lVert \vec{x}-\vec{x'}\rVert_2,\quad\mbox{for all}~ x.x'\in I\\
			\lVert \prodifs{F}(\vec{x},\vec{y})-\prodifs{F}(\vec{x},\vec{y'})\rVert_2\leq |\prodifs{b}|_{\infty}\lVert \vec{y}-\vec{y'}\rVert_2,\quad\mbox{for all}~ y,y'\in\mathbb{R}^m.
		\end{array} \right. 
		\label{cont.of F}
	\end{equation}
	where $|\prodifs{a}|_{\infty}=\max\{|a_{k,i_k}|:\myseries{1}{i_k}{N_k};~\myseries{1}{k}{m}\}<1$ and  $|\prodifs{b}|_{\infty}=\max\{|b_{k,i_k}|:\myseries{1}{i_k}{N_k};~\myseries{1}{k}{m}\}<1$.\\
	Finally, the maps $\prodifs{W}:I\times\mathbb{R}^m\rightarrow\prodifs{I}\times\mathbb{R}^m$ are defined by 
	
	\begin{equation}
		\prodifs{W}(\vec{x},\vec{y})=(\prodifs{L}(\vec{x}),\prodifs{F}(\vec{x},\vec{y}))\quad\mbox{for}~\myseries{1}{i_k}{N_k};~\myseries{1}{k}{m}.
	\end{equation}
	Then from (\ref{prodofL}) and (\ref{prodofF}), we have 
	\begin{equation}
		\prodifs{W}\left(\left(\prod_{k=1}^{m}x_{k,j_k}\right),\left(\prod_{k=1}^{m}y_{k,j_k}\right)\right)=\left(\left(\prod_{k=1}^{m}x_{k,\sigma(i_k,j_k)}\right),\left(\prod_{k=1}^{m}y_{k,\sigma(i_k,j_k)}\right)\right)
	\end{equation}
	for all $i_k\in \mathbb{N}_{N_k},~j_k\in \partial\mathbb{N}_{N_k,0},~ k\in\mathbb{N}_m$. This shows that the maps $\prodifs{W}$ for $\myseries{1}{i_k}{N_k};~\myseries{1}{k}{m}$, satisfies the join up conditions. Now, we see that the IFS
	\begin{equation}\label{noforpdifew}
		\{(I\times\mathbb{R}^m:\prodifs{W}):~\myseries{1}{i_k}{N_k};~\myseries{1}{k}{m}\}
	\end{equation}  is contractive. Let $\theta$ be a positive real number. We define a new norm $\norm{\cdot}{\theta}$ on $I\times\mathbb{R}^m$, which is equivalent to the norm $\lVert(\vec{x},\vec{y})\rVert=\norm{\vec{x}}{2}+\norm{\vec{y}}{2}$ on $I\times\mathbb{R}^m$, as follows:
	$$\norm{(\vec{x},\vec{y})}{\theta}=\norm{\vec{x}}{2}+\theta\norm{\vec{y}}{2}\quad\mbox{for all}~ (\vec{x},\vec{y})\in I\times\mathbb{R}^m.$$
	Then, we get the following.
	\begin{lm}
		If $\theta=\frac{\min\{(1-|\prodifs{c}|_{\infty}):~\myseries{1}{i_k}{N_k};\myseries{1}{k}{m}\}}{\max\{|\prodifs{a}|_{\infty}:~\myseries{1}{i_k}{N_k};\myseries{1}{k}{m}\}}$, $a=\max\{|\prodifs{c}|_{\infty}+\theta|\prodifs{a}|_{\infty}:~\myseries{1}{i_k}{N_k};\myseries{1}{k}{m}\}$ and $b=\max\{|\prodifs{b}|_{\infty}:~\myseries{1}{i_k}{N_k};\myseries{1}{k}{m}\}$, then the IFS $\{(I\times\mathbb{R}^m;\prodifs{W}):~\myseries{1}{i_k}{N_k};\myseries{1}{k}{m}\}$ defined above, associated with the data set $\{((\prodpoint{x}{i}),(\prodpoint{y}{i}))\in I\times\mathbb{R}^m:~\myseries{1}{i_k}{N_k};\myseries{1}{k}{m}\}$ is hyperbolic with respect to the norm $\norm{\cdot}{\theta}$ and with contractivity factor $s=\max\{a,b\}$.
	\end{lm}
	\begin{proof}
		Proof is similar to  Lemma~\ref{contlm}. But for the completeness, we prove it here. Let $\myvec{\vec{x}}{x}{m},\myvec{\vec{x'}}{x'}{m}\in I$ and $\myvec{\vec{y}}{y}{m},\myvec{\vec{y'}}{y'}{m}\in\mathbb{R}^m$. Then
		\begin{align*}
			&\norm{\prodifs{W}(\vec{x},\vec{y})-\prodifs{W}(\vec{x'},\vec{y'})}{\theta}\\
			&=\norm{(\prodifs{L}(\vec{x}),\prodifs{F}(\vec{x},\vec{y}))-(\prodifs{L}(\vec{x'}),\prodifs{F}(\vec{x'},\vec{y'}))}{\theta}\\
			&=\norm{\prodifs{L}(\vec{x})-\prodifs{L}(\vec{x'})}{2}+\theta\norm{\prodifs{F}(\vec{x},\vec{y})-\prodifs{F}(\vec{x'},\vec{y'})}{2}\\
			&\leq |\prodifs{c}|_{\infty}\norm{\vec{x}-\vec{x'}}{2}+\theta\left(\norm{\prodifs{F}(\vec{x},\vec{y})-\prodifs{F}(\vec{x'},\vec{y})}{2}+\norm{\prodifs{F}(\vec{x'},\vec{y})-\prodifs{F}(\vec{x'},\vec{y'})}{2}\right)\\
			&\leq |\prodifs{c}|_{\infty}\norm{\vec{x}-\vec{x'}}{2}+\theta|\prodifs{a}|\norm{\vec{x}-\vec{x'}}{2}+\theta|\prodifs{b}|_{\infty}\norm{\vec{y}-\vec{y'}}{2}\\
			&\leq \left(|\prodifs{c}|_{\infty}+\theta|\prodifs{a}|_{\infty}\right)\norm{\vec{x}-\vec{x'}}{2}+\theta|\prodifs{b}|_{\infty}\norm{\vec{y}-\vec{y'}}{2}\\
			&\leq a \norm{\vec{x}-\vec{x'}}{2}+\theta b \norm{\vec{y}-\vec{y'}}{2}\\
			&\leq s \left[\norm{\vec{x}-\vec{x'}}{2}+\theta\norm{\vec{y}-\vec{y'}}{2}\right]=s~\norm{(x,y)-(x',y')}{\theta}.
		\end{align*}
		Now, from given condition, we get $\theta<\frac{1-|\prodifs{c}|_{\infty}}{|\prodifs{a}|_{\infty}}$. So, $|\prodifs{c}|_{\infty}+\theta|\prodifs{a}|_{\infty}<1.$ Therefore, $a<1$. Also $b<1$. Hence $s<1$. This completes the proof.
	\end{proof}
	Since the IFS $\{(I\times\mathbb{R}^m;\prodifs{W}):~\myseries{1}{i_k}{N_k};\myseries{1}{k}{m}\}$ is contractive. Hence by Banach fixed point theorem it has an unique attractor $G$ (say). Now, we will prove that $G$ is the graph of a continuous product function from $I$ to $\mathbb{R}^m$, which interpolates the data set $\{((\prodpoint{x}{i}),(\prodpoint{y}{i}))\in I\times\mathbb{R}^m:~\myseries{1}{i_k}{N_k};\myseries{1}{k}{m}\}$.\\
	Let
	$$\mathscr{C}:=\left\{\profunc{f}:~f_k\in C[I_k]\quad\mbox{for}~\myseries{1}{k}{m}\right\},$$
	equip with the norm $\norm{\profunc{f}}{\infty}:=\max\left\{\norm{\profunc{f}(\vec{x})}{2}:~\vec{x}\in I\right\}$. Then the following result holds.
	\begin{thm}\label{rbnonfs}
		The norm $\norm{\cdot}{\infty}$ on $\mathscr{C}$ is equivalent to the norm $\norm{\cdot}{0}$, where the norm $\norm{\cdot}{0}$ on $\mathscr{C}$ is defined by 
		$$\norm{\profunc{f}}{0}:=\left\{\sum_{k=1}^{m}\norm{f_k}{\infty}^2\right\}^{\frac{1}{2}}.$$
		In particular, $$\frac{1}{\sqrt{m}}\norm{\profunc{f}}{0}\leq \norm{\profunc{f}}{\infty}\leq \norm{\profunc{f}}{0}.$$
	\end{thm}
	\begin{proof}
		Since $|f_k(x)|\leq \norm{f_k}{\infty}$ for all $x\in I_k$. Therefore,
		$$\norm{\profunc{f}(\vec{x})}{2}=\left\{\sum_{k=1}^{m}|{f_k}(x_k)|^2\right\}^{\frac{1}{2}}\leq \left\{\sum_{k=1}^{m}\norm{f_k}{\infty}^2\right\}^{\frac{1}{2}}\quad\mbox{for all}~\vec{x}\in I.$$
		So, by taking supremum over all $\vec{x}\in I$, we get
		\begin{equation} 
			\norm{\profunc{f}}{\infty}\leq \left\{\sum_{k=1}^{m}\norm{f_k}{\infty}^2\right\}^{\frac{1}{2}}=\norm{\profunc{f}}{0}. \label{1strel}
		\end{equation}
		To prove the other part, consider $\epsilon>0$, then by the definition of $\norm{f_j}{\infty}$, there exists $x_j\in I_j$, such that 
		$$\norm{f_j}{\infty}-\epsilon\leq |f_j(x_j)|<\norm{f_j}{\infty}\quad\mbox{for}~\myseries{1}{j}{m}.$$
		Let $\myvec{\vec{x}}{x}{m}$. Since $$|f_j(x_j)|\leq \left\{\sum_{k=1}^{m}|{f_k}(x_k)|^2\right\}^{\frac{1}{2}}=\norm{\profunc{f}(\vec{x})}{2}\leq \norm{\profunc{f}}{\infty}.$$ 
		Therefore, 
		$$\norm{f_j}{\infty}-\epsilon\leq \norm{\profunc{f}}{\infty}\quad\mbox{for all}~ \myseries{1}{j}{m}.$$
		Hence
		$$\left\{\sum_{k=1}^{m}(\norm{f_k}{\infty}-\epsilon)^2\right\}^{\frac{1}{2}}\leq \sqrt{m}\norm{\profunc{f}}{\infty}.$$
		This is true for all $\epsilon>0$ . Hence by taking $\epsilon\rightarrow 0$ in the above expression, we get 
		\begin{equation}
			\frac{1}{\sqrt{m}}\norm{\profunc{f}}{0}\leq \norm{\profunc{f}}{\infty}. \label{2ndrel}
		\end{equation}
		Hence from (\ref{1strel}) and (\ref{2ndrel}), we get the required result.
	\end{proof}
	
	Now, by Theorem~\ref{rbnonfs}, we see that $(\mathscr{C},\norm{\cdot}{\infty})$ is a complete space. We define the closed subspace $$\mathscr{T}:=\left\{\profunc{f}:~f_k\in C_{y_{k,0},y_{k,N_k}}[I_k];~\myseries{1}{k}{m} \right\}\quad\mbox{of}~(\mathscr{C},\norm{\cdot}{\infty}).$$ 
	
	The following theorem shows the existence of the product fractal function.
	\begin{thm}
		The attractor $G$ of the IFS $\{(I\times\mathbb{R}^m;\prodifs{W}):~\myseries{1}{i_k}{N_k};\myseries{1}{k}{m}\}$, associated with the data set $\{((\prodpoint{x}{i}),(\prodpoint{y}{i}))\in I\times\mathbb{R}^m:~\myseries{1}{i_k}{N_k};\myseries{1}{k}{m}\}$ is the graph of the continuous product function $\profunc{g}: I\rightarrow\mathbb{R}^m$, where $\profunc{g}\in\mathscr{T}$. \label{graphofcpf}
	\end{thm}
	\begin{proof}
		Let us define the RB-operator $T:\mathscr{T}\rightarrow\mathscr{T}$ as follows: 
		\begin{equation}
			T(\profunc{f})(\vec{x})=\prodifs{F}\left(\prodifs{L^{-1}}(\vec{x}),\profunc{f}(\prodifs{L^{-1}}(\vec{x}))\right),\label{Rbop}
		\end{equation}
		where $\vec{x}\in\prodifs{I}$. Now, we prove that $T$ is well-defined. Since from the definition of $\prodifs{L}$, we get $$\prodifs{L^{-1}}(\vec{x})=\left(L_{1,i_1}^{-1}(x_1),L_{2,i_2}^{-1}(x_2),\ldots,L_{m,i_m}^{-1}(x_m)\right)\quad\mbox{for all}~ \myvec{\vec{x}}{x}{m}\in\prodifs{I}.$$
		Therefore,
		\begin{align*}
			T(\profunc{f})(\vec{x})=&\prodifs{F}\left(\prod_{k=1}^{m}L_{k,i_k}^{-1}(x_k),\prod_{k=1}^{m}f_k(L_{k,i_k}^{-1}(x_k))\right)\\
			=&\prod_{k=1}^{m}F_{k,i_k}\left(L_{k,i_k}^{-1}(x_k),f_k(L_{k,i_k}^{-1}(x_k))\right),\quad[\mbox{using (\ref{prod of F}})].
		\end{align*}
		Let $\myvec{\vec{x}}{x}{m}\in\partial\prodifs{I}$, where $\partial\prodifs{I}$, denotes the boundary of $\prodifs{I}$. Now, $\vec{x}$ may be anywhere on the boundary $\partial\prodifs{I}$ but we  deal with the case, where it is an intersection point. Let $\vec{x}\in I_{i_1i_2\cdots i_p\cdots,i_m}\cap I_{i_1i_2\cdots i_p+1\cdots,i_m}\cap I_{i_1i_2\cdots i_q\cdots,i_m}\cap I_{i_1i_2\cdots i_q+1\cdots,i_m}\cap\cdots$(up to a finite number of intersection), where $i_p\in\{1,2,\ldots,N_p-1\}$ and $i_q\in\{1,2,\ldots,N_q-1\}$. For simplicity, let $x_p\in I_{p,i_p}\cap I_{p,i_p+1}$ and $x_q\in I_{q,i_q}\cap I_{q,i_q+1}$. This is possible if only if $x_p=x_{p,i_p}$ and $x_q=x_{q,i_q}$. Without loss of generality, we assume that $p<q$. Here we  deal with the following four possible cases.\\ 
		
		Case 1. If $x_{p,i_p}\in I_{p.i_p}$ and $x_{q,i_q}\in I_{q,i_q}$, then
		\begin{align*}
			T(\profunc{f}&)(\vec{x})=\left(F_{1,i_1}(L_{1,i_1}^{-1}(x_1),f_1(L_{1,i_1}^{-1}(x_1))),\ldots,F_{p,i_p}(L_{p,i_p}^{-1}(x_{p,i_p}),f_p(L_{p,i_p}^{-1}(x_{p,i_p}))),\right.\\ &\left.\ldots,F_{q,i_q}(L_{q,i_q}^{-1}(x_{q,i_q}),f_q(L_{q,i_q}^{-1}(x_{q,i_q}))),\ldots,F_{m,i_m}(L_{m,i_m}^{-1}(x_m),f_m(L_{m,i_m}^{-1}(x_m)))\right)\\
			&=\left(F_{1,i_1}\left(L_{1,i_1}^{-1}(x_1),f_1(L_{1,i_1}^{-1}(x_1))\right),\ldots,y_{p,i_p},\right.\\
			&\left.\ldots,y_{q,i_q},\ldots,F_{m,i_m}\left(L_{m,i_m}^{-1}(x_m),f_m(L_{m,i_m}^{-1}(x_m))\right)\right).
		\end{align*}  
		Case 2. If $x_{p,i_p}\in I_{p.i_p}$ and $x_{q,i_q}\in I_{q,i_q+1}$, then   
		\begin{align*}
			T(\profunc{f}&)(\vec{x})=\left(F_{1,i_1}(L_{1,i_1}^{-1}(x_1),f_1(L_{1,i_1}^{-1}(x_1))),\ldots,F_{p,i_p}(L_{p,i_p}^{-1}(x_{p,i_p}),f_p(L_{p,i_p}^{-1}(x_{p,i_p}))),\right.\\ &\left.\ldots,F_{q,i_q+1}(L_{q,i_q+1}^{-1}(x_{q,i_q}),f_q(L_{q,i_q+1}^{-1}(x_{q,i_q}))),\ldots,F_{m,i_m}(L_{m,i_m}^{-1}(x_m),f_m(L_{m,i_m}^{-1}(x_m)))\right)\\
			&=\left(F_{1,i_1}\left(L_{1,i_1}^{-1}(x_1),f_1(L_{1,i_1}^{-1}(x_1))\right),\ldots,y_{p,i_p},\right.\\
			&\left.\ldots,y_{q,i_q},\ldots,F_{m,i_m}\left(L_{m,i_m}^{-1}(x_m),f_m(L_{m,i_m}^{-1}(x_m))\right)\right).
		\end{align*} 
		Case 3. If $x_{p,i_p}\in I_{p.i_p+1}$ and $x_{q,i_q}\in I_{q,i_q}$, then
		\begin{align*}
			T(\profunc{f}&)(\vec{x})=\left(F_{1,i_1}(L_{1,i_1}^{-1}(x_1),f_1(L_{1,i_1}^{-1}(x_1))),\ldots,F_{p,i_p+1}(L_{p,i_p+1}^{-1}(x_{p,i_p}),f_p(L_{p,i_p+1}^{-1}(x_{p,i_p}))),\right.\\ &\left.\ldots,F_{q,i_q}(L_{q,i_q}^{-1}(x_{q,i_q}),f_q(L_{q,i_q}^{-1}(x_{q,i_q}))),\ldots,F_{m,i_m}(L_{m,i_m}^{-1}(x_m),f_m(L_{m,i_m}^{-1}(x_m)))\right)\\
			&=\left(F_{1,i_1}\left(L_{1,i_1}^{-1}(x_1),f_1(L_{1,i_1}^{-1}(x_1))\right),\ldots,y_{p,i_p},\right.\\
			&\left.\ldots,y_{q,i_q},\ldots,F_{m,i_m}\left(L_{m,i_m}^{-1}(x_m),f_m(L_{m,i_m}^{-1}(x_m))\right)\right).
		\end{align*}   
		Case 4. If $x_{p,i_p}\in I_{p.i_p+1}$ and $x_{q,i_q}\in I_{q,i_q+1}$, then 
		\begin{align*}
			T(\profunc{f}&)(\vec{x})=\left(F_{1,i_1}(L_{1,i_1}^{-1}(x_1),f_1(L_{1,i_1}^{-1}(x_1))),\ldots,F_{p,i_p+1}(L_{p,i_p+1}^{-1}(x_{p,i_p}),f_p(L_{p,i_p+1}^{-1}(x_{p,i_p}))),\right.\\ &\left.\ldots,F_{q,i_q+1}(L_{q,i_q+1}^{-1}(x_{q,i_q}),f_q(L_{q,i_q+1}^{-1}(x_{q,i_q}))),\ldots,F_{m,i_m}(L_{m,i_m}^{-1}(x_m),f_m(L_{m,i_m}^{-1}(x_m)))\right)\\
			&=\left(F_{1,i_1}\left(L_{1,i_1}^{-1}(x_1),f_1(L_{1,i_1}^{-1}(x_1))\right),\ldots,y_{p,i_p},\right.\\
			&\left.\ldots,y_{q,i_q},\ldots,F_{m,i_m}\left(L_{m,i_m}^{-1}(x_m),f_m(L_{m,i_m}^{-1}(x_m))\right)\right).
		\end{align*}
		From above, we see that the values of $T(\profunc{f})(\vec{x})$ are same in all the four cases. Similarly, all other possibilities can be checked to conclude that $T(\profunc{f})(\vec{x})$ is well
		defined on the boundary of $\prodifs{I}$. Also, from (\ref{prodrbop}), we have
		\begin{align*}
			T(\profunc{f})(\vec{x})=\prod_{k=1}^{m}F_{k,i_k}\left(L_{k,i_k}^{-1}(x_k),f_k(L_{k,i_k}^{-1}(x_k))\right)=\prod_{k=1}^{m}(T_{k}f_k)(x_k).
		\end{align*} 
		Therefore,
		\begin{equation}
			T(\profunc{f})=T_1f_1\times T_2f_2\times\cdots\times T_mf_m.
			\label{pofrboper}
		\end{equation}
		Hence $T$ is well-defined. Also, if $\vec{x}=(\prodpoint{x}{i})\in \prodifs{I}$, then
		\begin{align*}
			T(\profunc{f})(\vec{x})=&\prod_{k=1}^{m}F_{k,i_k}\left(L_{k,i_k}^{-1}(x_{k,i_k}),f_k(L_{k,i_k}^{-1}(x_{k,i_k}))\right)=\prod_{k=1}^{m}F_{k,i_k}\left(x_{k,0},f_k(x_{k,0})\right)\\
			=&\prod_{k=1}^{m}F_{k,i_k}\left(x_{k,0},y_{k,0}\right)=\prod_{k=1}^{m}y_{k,i_k}.
		\end{align*}
		So, $T(\profunc{f})$ interpolates data points  $$\{((\prodpoint{x}{i}),(\prodpoint{y}{i}))\in I\times\mathbb{R}^m:~\myseries{1}{i_k}{N_k};\myseries{1}{k}{m}\}.$$
		Next, we prove that $T$ is a contraction map on $\mathscr{T}$.
		For that, let $\profunc{f},\profunc{f'}\in\mathscr{T}$, and $\myvec{\vec{x}}{x}{m}\in\prodifs{I}$. Then from (\ref{cont.of F}) and (\ref{Rbop}), we get
		\begin{align*}
			\norm{T(\profunc{f})&(\vec{x})-T(\profunc{f'})(\vec{x})}{2}\\ &\leq|\prodifs{b}|_{\infty}\norm{\profunc{f}(\prodifs{L^{-1}}(\vec{x}))-\profunc{f'}(\prodifs{L^{-1}}(\vec{x}))}{2}\\
			&\leq |\prodifs{b}|_{\infty}\norm{\profunc{f}-\profunc{f'}}{\infty}\\
			&\leq b \norm{\profunc{f}-\profunc{f'}}{\infty},
		\end{align*}
		where  $b=\max\{|\prodifs{b}|_{\infty}:~\myseries{1}{i_k}{N_k};\myseries{1}{k}{m}\}<1$. Now, by taking supremum over all $\vec{x}\in I$, we get
		$$\norm{T(\profunc{f})-T(\profunc{f'})}{\infty}\leq b \norm{\profunc{f}-\profunc{f'}}{\infty}.$$
		This shows that $T$ is a contraction map on the complete space $\mathscr{T}$. Hence by Banach fixed point theorem, $T$ has an unique fixed point $\profunc{g}$ in $\mathscr{T}$. Finally, we see that $G$ is the graph of the function $\profunc{g}$. Let $\tilde{G}=\left\{\left(\vec{x},\profunc{g}(\vec{x})\right):~\vec{x}\in I\right\}$. Then
		\begin{align*}
			&\bigcup\left\{\prodifs{W}(\tilde{G}):~\myseries{1}{i_k}{N_k};\myseries{1}{k}{m}\right\}\\
			&=\bigcup \left\{\prodifs{W}((\vec{x},\profunc{g}(\vec{x})):~\vec{x}\in I,~\myseries{1}{i_k}{N_k};\myseries{1}{k}{m} \right\}\\
			&=\bigcup\left\{\left(\prodifs{L}(\vec{x}),\prodifs{F}(\vec{x},\profunc{g}(\vec{x}))\right):~\vec{x}\in I,~\myseries{1}{i_k}{N_k};\myseries{1}{k}{m}\right\}.
		\end{align*}
		Now, since $\profunc{g}$ is the fixed point of $T$. Hence from (\ref{Rbop}), we get 
		$$\profunc{g}(\prodifs{L}(\vec{x}))=\prodifs{F}(\vec{x},\profunc{g}(\vec{x})).$$
		Therefore,
		\begin{align*}
			&\bigcup\left\{\prodifs{W}(\tilde{G}):~\myseries{1}{i_k}{N_k};\myseries{1}{k}{m}\right\}\\
			&=\bigcup\left\{\left(\prodifs{L}(\vec{x}),\profunc{g}(\prodifs{L}(\vec{x}))\right):~\vec{x}\in I,~\myseries{1}{i_k}{N_k};\myseries{1}{k}{m}\right\}\\	
			&=\left\{\left(\vec{x},\profunc{g}(\vec{x})\right):~\vec{x}\in I\right\}\\
			&=\tilde{G}.
		\end{align*}
		Hence by uniqueness of the attractor, we get $G=\tilde{G}$.
	\end{proof}
	\begin{dfn}
		The function $\profunc{g}$ is called the product fractal interpolation function (PFIF) corresponding to the IFS \ref{noforpdifew}, which interpolates the data set\\ $\{((\prodpoint{x}{i}),(\prodpoint{y}{i}))\in I\times\mathbb{R}^m:~\myseries{1}{i_k}{N_k};\myseries{1}{k}{m}\}$.
	\end{dfn}
	\begin{cor}
		If $f_k$ are the FIF corresponding to the IFS $\{(I_k\times\mathbb{R};W_{k,i_k}):\myseries{1}{i_k}{N_k}\}$ which interpolates the data sets $\{(x_{k,i_k},y_{k,i_k})\in\mathbb{R}^2:\myseries{0}{i_k}{N_k}\}$ for $k=1,2,\ldots,m$, then $\profunc{f}$ is the PFIF corresponding to the IFS $\{(I\times\mathbb{R}^m;\prodifs{W}):~\myseries{1}{i_k}{N_k};\myseries{1}{k}{m}\}$ which interpolates the data set $\{((\prodpoint{x}{i}),(\prodpoint{y}{i}))\in I\times\mathbb{R}^m:~\myseries{1}{i_k}{N_k};\myseries{1}{k}{m}\}$. \label{proditerfs}
	\end{cor}
	\begin{proof}
		By the given condition, $T_kf_k=f_k$, for $k=1,2,\ldots,m$. Hence from (\ref{pofrboper}), we have 
		$$T(\profunc{f})=T_1f_1\times T_2f_2\times\cdots\times T_mf_m=\profunc{f}.$$ 
		Therefore, by uniqueness of the fixed point it follows that
		$$\profunc{f}=\profunc{g}.$$
	\end{proof}

\end{document}